\documentclass[onefignum,onetabnum]{siamart190516}

\usepackage{lipsum}
\usepackage{amsmath,amsfonts,amssymb,amscd}
\usepackage{graphicx}
\usepackage{subcaption}
\usepackage{enumerate}
\usepackage{graphicx}
\usepackage{listings}
\usepackage{epstopdf}
\usepackage{algorithmic}
\usepackage{xcolor}
\usepackage{tikz}
\usepackage{pgfkeys}
\usepackage{adjustbox}
\usepackage{footnotehyper}
\usetikzlibrary{shapes,arrows}
\ifpdf
  \DeclareGraphicsExtensions{.eps,.pdf,.png,.jpg}
\else
  \DeclareGraphicsExtensions{.eps}
\fi

\newsiamremark{remark}{Remark}
\newsiamremark{hypothesis}{Hypothesis}
\crefname{hypothesis}{Hypothesis}{Hypotheses}
\newsiamthm{claim}{Claim}

\headers{Systematic Lyapunov analyses of continuous-time models}{C. Moucer, A. Taylor, F. Bach}

\title{A systematic approach to Lyapunov analyses of continuous-time models in convex optimization 
}

\author{Céline Moucer\footnotemark[2]
\thanks{Ecole Nationale des Ponts et Chaussées, Marne-la-Vallée, France.}
\and Adrien Taylor\thanks{DI ENS, \'Ecole normale supérieure, Université PSL, CNRS, INRIA, 75005 Paris, France (\email{celine.moucer@inria.fr}), (\email{adrien.taylor@inria.fr}), (\email{francis.bach@inria.fr}).}
\and Francis Bach\footnotemark[2].}

\usepackage{amsopn}

\lstdefinestyle{style1}{
    basicstyle=\ttfamily,
    keywordstyle =    \color{black}, 
    keywordstyle =    \color{black}, 
    commentstyle =    \color{black},
    stringstyle  =    \color{black},
    columns=fullflexible,
    keepspaces=true,
    upquote=true,
}

\usepackage{pgfplots}
\pgfplotsset{compat=1.15}

\tikzstyle{loosely dashed}=          [dash pattern=on 2pt off 4pt]
	
\pgfplotsset{plotOptions1/.style={%
		width=\linewidth,
				ymax=2.5,ymin=0.0001,
				xmin=0.0001,xmax=1.5,
		xlabel={Strong convexity parameter $\mu$},
		ylabel={Worst-case guarantee},
		label style={font=\footnotesize},
		legend style={font=\scriptsize},
		ytick={0.0001, 0.001, 0.01, 0.1, 1},
		xtick={0.0001, 0.001, 0.01, 0.1, 1},
		tick label style={font=\footnotesize},
		solid,
		very thick
	}}

\pgfplotsset{plotOptions2/.style={%
		width=\linewidth,
				ymax=5,ymin=-0.5,
				xmin=-10,xmax=10,
		xlabel={Step-size},
		ylabel={Worst-case guarantee},
		label style={font=\footnotesize},
		legend style={font=\scriptsize},
		ytick={0, 1, 2, 3, 4, 5},
		xtick={-10, -5,   0, 5, 10},
		tick label style={font=\footnotesize},
		solid,
		very thick
	}}
	
\pgfplotsset{plotOptions3/.style={%
		width=\linewidth,
				ymax=2.1,ymin=0.001,
				xmin=0.0001,xmax=1.5,
		xlabel={Strong convexity parameter $\mu$},
		ylabel={Worst-case guarantee},
		label style={font=\footnotesize},
		legend style={font=\scriptsize},
		ytick={0.001, 0.01, 0.1, 1},
		xtick={0.0001, 0.001, 0.01, 0.1, 1},
		tick label style={font=\footnotesize},
		solid,
		very thick
	}}
	
\pgfplotsset{plotOptions4/.style={%
		width=\linewidth,
				ymax=2.1,ymin=0.000005,
				xmin=0.0001,xmax=1.5,
		xlabel={Strong convexity parameter $\mu$},
		ylabel={Worst-case guarantee},
		label style={font=\footnotesize},
		legend style={font=\scriptsize},
		ytick={0.00001, 0.0001, 0.001, 0.01, 0.1, 1},
		xtick={0.0001, 0.001, 0.01, 0.1, 1},
		tick label style={font=\footnotesize},
		solid,
		very thick
	}}

\ifpdf
\hypersetup{
  pdftitle={Systematic continuous-time analyses for first-order optimization},
  pdfauthor={C. Moucer, A Taylor, F. Bach},
  backref=page
}
\fi

\begin{document}

\maketitle
 
\begin{abstract}
First-order methods are often analyzed via their continuous-time models, where their worst-case convergence properties are usually approached via Lyapunov functions. In this work, we provide a systematic and principled approach to find and verify Lyapunov functions for classes of ordinary and stochastic differential equations. More precisely, we extend the performance estimation framework, originally proposed by Drori and Teboulle~\cite{drori2012performance}, to continuous-time models. We retrieve convergence results comparable to those of discrete-time methods using fewer assumptions and inequalities, and provide new results for a family of stochastic accelerated gradient flows.
\end{abstract}




\section{Introduction}

Convex optimization is an important tool in the numerical analyst toolbox. It serves, among others, for framing modeling problems in data science and signal processing. We consider optimization problems of the form:
\begin{equation}
    \min_{x \in \mathbf{R}^d} f(x),
    \label{eq:optim}
\end{equation}
where $f$ is convex and differentiable. First-order methods (that gather information about $f$ by evaluating its gradient at past iterates) are very popular to solve these problems, due to their attractive low cost per iteration, and to the fact that data science applications typically do not require very accurate solutions~\cite{bottou2007tradeoffs}. Gradient descent is a common first-order method, which starts from a point $x_0 \in \mathbf{R}^d$ and whose iterates are given by the simple recursion
\begin{equation}
    x_{k+1} = x_k - \gamma \nabla f(x_k),
    \label{eq:gd}
\end{equation}
where $\gamma>0$ is a step size. Gradient descent with small step sizes is directly related to the so-called gradient flow:
\begin{equation}
\label{eq:gf_init}
\begin{aligned}
   \dot{X}_t &= -\nabla f(X_t), \ X_0=x_0 \in \mathbf{R}^d,
\end{aligned}
\end{equation}
with the notation $\dot{X}_t \triangleq \frac{d}{dt}X_t$ and where the solution $X_t$ of the ordinary differential equation (ODE) verifies $X_{t_k} \approx x_k$ with the identification $t_k = \gamma k$. In numerical integration, gradient descent~\eqref{eq:gd} is also known as the explicit Euler scheme for integrating gradient flows. Recently, Su et~al.~\cite{Su2015} have interpreted Nesterov's accelerated gradient~\cite{nesterov1983} in a similar fashion through its continuous-time version, paving the way to several continuous-time analyses of accelerated methods~\cite{shi2018understanding, wilson2016lyapunov, wibisono2016variational}.
\smallbreak
Many applications entail some randomness and require a stochastic modeling of the function $f$, which is often defined in terms of an expectation $f(x) = \mathbf{E}_\xi[\tilde{f}(x, \xi)]$. The function $f$ is the expectation over some random variable $\xi$, and accounts for some stochastic modeling. When $\xi$ is drawn uniformly from a finite set of possible samples $(\xi_1, ..., \xi_n)$, we have a finite sum $f(x) = \frac{1}{n}\sum_{k=1}^n\tilde{f}(x, \xi_{k})$. As soon as the number of data points $n$ is large, computing the gradient of a finite sum, as it is done in gradient-based methods, is possibly expensive (computing the gradient of each element of the sum, which is possibly very large). Stochastic gradient descent (SGD) provides an alternative with lower computational burden per iteration, by evaluating only the gradient of a single $\tilde{f}(\cdot, \xi_{i_k})$ per iteration:
\begin{equation*}
    x_{k+1} = x_k - \gamma \nabla \tilde{f}(x_k, \xi_{i_k}),
\end{equation*}
where $\gamma>0$ is the step size, $\xi_{i_k}$ is drawn uniformly at random in $(\xi_1, ..., \xi_n)$. Thereby $\nabla \tilde{f}(x_k, \xi_{i_k})$ is an unbiased estimate of the full gradient: $\mathbf{E}_{i_k}[\nabla \tilde{f}(x_k, \xi_{i_k})] = \nabla f(x_k)$. Li et~al.~\cite{li2017stochastic, Li2019} derived stochastic differential equations (SDE) approximating~SGD:
\begin{equation*}
    dX_t = -\nabla f(X_t)dt + \sigma(X_t)dB_t,
\end{equation*}
where $\sigma(X_t)$ is a noise parameter connected to parameters of the method, and these were further developed by Shi et~al.~\cite{shi2020learning, xu2018}. Relying on approximate theorems between SDEs and original stochastic gradient algorithms, SDEs have thus become a tool for analyzing convergence speeds of discrete-time methods. Usually, gradient flows (resp. first-order methods) are studied via worst-case convergence properties, that hold for any function of a given class, and any trajectory generated by the ODE (resp. optimization method). In many cases, continuous-time approaches seem to allow for shorter, simpler, and thereby more intuitive proofs. They also bring insights on what can be expected from optimization methods. 
\smallbreak
The analysis of continuous-time models often relies on Lyapunov stability arguments, as in system theory and physics, where energy dissipation plays a crucial role. The existence of such Lyapunov functions provides direct convergence proofs for ODEs under consideration. The main challenge in the Lyapunov approach is to find a suitable function that is decreasing along all trajectory generated by an~ODE.
\smallbreak
From an outsider point of view, these analyses are often seen as complicated and technical to reach. In this work, we remedy this problem by extending the systematic approach based on semidefinite programming (SDP) originally coined by Drori and Teboulle~\cite{drori2012performance} for certifying convergence of optimization methods. This technique is referred as ``performance estimation problems'' (PEPs). The main contribution of this work is to provide a tool for analyzing convergence of continuous-time models, by constructing Lyapunov functions suited to a gradient-based ODE in a systematic way, using small-sized SDPs reformulations. Furthermore, this procedure benefits from tightness properties, meaning that the feasibility of the SDP allows to conclude that there exists a Lyapunov function within the prescribed family of Lyapunov functions under consideration. Reciprocally, infeasibility of the SDP allows to conclude that there exists no such valid Lyapunov function within the family.

\subsection{Lyapunov functions}

Lyapunov functions are a standard tool for dealing with convergence properties of gradient flows. Such functions are also more generally used for studying stability properties of dynamical systems~\cite{Kalman1960}. More precisely, consider a differentiable function $f$ within a class $\mathcal{F}$, and an ODE aiming at minimizing $f$. For such a continuous-dynamical system with a stationary point $x_\star \in \text{argmin}_x f(x)$, we call $\mathcal{V}: \mathbf{R}^d \times \mathbf{R}^+ \rightarrow \mathbf{R}$ a Lyapunov function if it is differentiable and satisfies the following conditions for all trajectory $X_t$ generated by the ODE:
\begin{itemize}
    \item $\mathcal{V}(x, t)=0 \iff x = x_\star$,
    \item $\mathcal{V}(X_t, t) \geqslant 0$,
    \item $\frac{d}{dt}\mathcal{V}(X_t, t) \leqslant 0$,
\end{itemize}
for all $t\geqslant 0$.
\smallbreak
Lyapunov functions are suited for deriving both linear (or exponential) and sublinear convergence rates. When looking for linear convergence rates (as we may expect for strongly convex functions), we typically use $\frac{d}{dt}\mathcal{V}(X_t) \leqslant -\tau \mathcal{V}(X_t)$ instead of the third condition, where $\tau$ depends on the class of functions and on the ODE (more details Remark~\ref{remark_lyap}). In this ad-hoc definition, we enforced nonnegativity along the trajectory $X_t$, but definitions often require nonnegativity on $\mathbf{R}^d$~\cite{taylor2018lyapunov}. There exist similar definitions of Lyapunov functions for discrete-time optimization methods~\cite{taylor2018lyapunov, sanzserna2021connections, lessard2016analysis}
\smallbreak
With this approach, convergence guarantees highly depend on how rich is the family of Lyapunov functions under consideration. In this work, we use a family of quadratic Lyapunov functions that is popular and natural for studying both discrete-time~\cite{nesterov1983, 2021taylordaspremont} and continuous-time~\cite{Su2015} optimization schemes.

\subsection{Prior works}
Lyapunov functions are common for analyzing continuous-time and discrete-time models in convex optimization. For example, convergence proofs for Nesterov's accelerated gradient method typically rely on such Lyapunov analyses~\cite{nesterov1983}~\cite[Theorem 4.8]{2021taylordaspremont}. In the recent~\cite{bansal2019potential}, the authors proposed Lyapunov-based analyses for many first-order methods, for linear and sublinear convergence rates. Continuous-time versions of optimization methods also often involve Lyapunov arguments, such as Nesterov's accelerated gradient flow introduced in~\cite{Su2015}, and its high-resolution ODEs for strongly convex functions proposed in~\cite{shi2018understanding}, or accelerated mirror descent whose continuous-time dynamics was analyzed in~\cite{Krichene2015}. 
\smallbreak
Different techniques were developed to compute suitable Lyapunov functions. The authors of~\cite{wibisono2016variational} put forward an approach based on Bregman Lagrangian for accelerated methods in potentially non-Euclidean settings, further developed in~\cite{wilson2016lyapunov}.~\cite{Diakonikolas2021} directly derived Lyapunov functions from Hamiltonian equations describing dynamics of ODEs. Using similar conservation laws in a dilated coordinate system,~\cite{suh2022continuoustime} also generated Lyapunov functions in a principled way.
\smallbreak
Given a class of functions and an optimization method, proving a convergence rate mostly consists in combining inequalities characterizing the class of functions at hand. Recently, the automated search for combination of inequalities formulated as semidefinite programs was pioneered by Drori and Teboulle~\cite{drori2012performance}, and led to the notion of performance estimation problems. Their work was followed up in~\cite{taylor2016smooth,2017taylor} to provide worst-case bounds in a principled way, and extended to the Lyapunov framework~\cite{taylor2018lyapunov}. A competing strategy inspired by control theory was developed by~\cite{lessard2016analysis,hu2017dissipativity}, where Lyapunov functions for discrete-time models are constructed using integral quadratic constraints (IQCs) and semidefinite programming; a similar approach was applied to continuous-time models in~\cite{fazlyab2018analysis}. Connections between Lyapunov functions obtained via the IQC framework in continuous-time and discrete-time were highlighted by~\cite{sanzserna2021connections}.
\smallbreak
For stochastic differential equations (SDEs), convergence proofs can also be obtained through the Lyapunov approach, together with Ito's calculus. For some well-chosen Lyapunov functions,~\cite{orvieto2020continuoustime} analyzed both SGD, SAGA~\cite{defazio2014}, and SVRG~\cite{Johnson2013}.~\cite{xu2018, xu2018strongly} extended the framework of Bregman Lagrangian Lyapunov functions~\cite{wibisono2016variational} to the stochastic setting. To the best of our knowledge, a systematic way of verifying a Lyapunov function for SDEs has not been developed yet.

\subsection{Contributions and organization}
In this work, we are concerned with worst-case convergence analyses of ODEs and SDEs for modeling (stochastic) gradient based optimization methods. We propose a principled approach to worst-case analyses based on Lyapunov functions, SDPs and Ito's calculus.
\smallbreak
In Section~\ref{sec:deterministic}, we extend the performance estimation approach developed for optimization methods to gradient flows, which originates from a (possibly strongly) convex function. In short, we find Lyapunov functions as feasible points of certain linear matrix inequalities (LMIs). All codes for reproducing numerical results can be found at \url{https://github.com/CMoucer/PEP_ODEs}. Following this work, we have also added continuous-time models to the Python package PEPit~\cite{goujaud2022}. After that, building on the first part of this work for gradient-based ODEs, we analyze continuous-time versions of stochastic optimization algorithms. 
\smallbreak
Section~\ref{sec:sde_for_sge} studies properties of trajectories generated by SDEs that approximates stochastic gradient methods. We obtain a simple version of the trade-off between forgetting the initial conditions and diminishing the noise, with and without averaging. It appears that decreasing step sizes, together with a nonuniform version of averaging, allows to reach an optimal trade-off between the two terms. Our results match those obtained for the stochastic gradient method, but with more compact analyses than those of the discrete-time setting.
\smallbreak
In Section~\ref{sec:extensions}, we prove that stochastic accelerated gradient flows require diminishing step sizes to converge in our setting. In contrast to first-order stochastic gradient flows, averaging does not preserve convergence for SDEs approximating accelerated gradient methods with constant step size. 

\subsection{Assumptions}
\label{sec:assumptions}
Throughout this work, functions to be minimized are convex (see Problem~\eqref{eq:optim}). Under this assumption, stationary points are global minimizers. We restrict ourselves to continuous-time versions of gradient descent, accelerated gradient descent and stochastic gradient descent.
\smallbreak
Let us recall a few basic definitions and properties characterizing the classes of functions under consideration within the next sections. A function $f : \mathbf{R}^d \rightarrow \mathbf{R}$ is convex if for all $x,y \in \mathbf{R}^d$, and for all $\lambda \in [0,1]$, $f(\lambda x + (1-\lambda)y) \leqslant \lambda f(x) + (1-\lambda)f(y)$. Such functions (i.e. with full domain, $\text{dom}(f) = \mathbf{R}^d$) are convex closed proper (CCP) (i.e., their epigraphs are non-empty closed convex sets). For simplicity, we assume in addition differentiability of $f$.  Such a differentiable function $f$ is convex if and only if for all $x,y \in \mathbf{R}^d$, $f(y) \geqslant f(x) + \langle \nabla f(x), y-x \rangle$. A differentiable function $f$ is $L$-smooth if its gradient is $L$-Lipschitz, that is if for any $x,y  \ \in \mathbf{R}^d$:
\begin{equation*}
        \|\nabla f(x) - \nabla f(y)\| \leqslant L \|x-y\|.
\end{equation*}
Smoothness is a common assumption for analyzing optimization methods, that limits the growth rate of the function. A convex differentiable function $f$ is $\mu$-strongly convex if for any $x, y \in \mathbf{R}^d$ it satisfies:
\begin{equation*}
         \|\nabla f(x) - \nabla f(y)\| \geqslant \mu \|x-y\|.
\end{equation*}
Strong convexity ensures the function is not too flat, and the unicity of the minimizer~$x_\star$. We denote by $\mathcal{F}_{\mu,L}$ the family of $L$-smooth $\mu$-strongly convex functions from $\mathbf{R}^d$ to $\mathbf{R}$, with $0 \leqslant \mu \leqslant L \leqslant +\infty$. Weaker assumptions than strong convexity are also encountered in the literature for analyzing gradient algorithms, and leads to similar convergence guarantees. Among them, Bolte~et~al.~\cite[Appendix 5.2]{bolte2009} proved linear convergence of gradient descent under Łojasiewicz inequality, first introduced by Łojasiewicz~\cite{Lojasiewicz1963}. Other relaxed versions of strong convexity followed~\cite{necoara2019linear}.

\section{A principled approach to Lyapunov functions for gradient flows}
\label{sec:deterministic}

In this section, we study convergence properties of the gradient flow and its accelerated versions, via quadratic Lyapunov functions. In this context, we show that verifying such a Lyapunov function can be formulated as verifying the feasibility of an LMI. This framework allows to search for Lyapunov functions, and to derive convergence bounds for non-autonomous gradient flows.

\subsection{The gradient flow} Let us consider the gradient flow: 
\begin{equation*}
    \begin{aligned}
    \dot{X_t} &= - \nabla f (X_t), \ 
    X_0 = x_0 \in \mathbf{R}^d.
    \end{aligned}
\end{equation*}
Let $x_\star$ be a global minimizer of $f$. Without further assumptions, the function $f$ is decreasing along the trajectory $X_t$ solution to the gradient flow. The Lyapunov function $\mathcal{V}(X_t) = f(X_t) - f(x_\star)$ is indeed nonnegative, equal to zero at $x_\star$ and has a nonpositive derivative with respect to time $\frac{d}{dt}\mathcal{V}(X_t) = \dot{X_t}^\top \nabla f(X_t) = -\|\nabla f(X_t)\|^2$.
\smallbreak

In the next section, we show how to obtain and verify such Lyapunov functions, and their corresponding convergence rates in the case of gradient flow originating from a strongly convex function.

\subsubsection{Minimizing strongly convex functions}
\label{sec:methodo}
Let $f$ be $\mu$-strongly convex (i.e., $f\in\mathcal{F}_{\mu,\infty}$) with $\mu>0$, and $x_\star$ be its unique minimizer such that $f(x_\star) = f_\star$. In this context, it is possible to prove linear convergence of the gradient flow to its stationary point. Scieur et~al.~proved in~\cite[Proposition 1.1]{scieur2017integration} the following convergence bound in function values for the gradient flow:
\begin{equation}
\label{reflyapmu}
f(X_t) - f_\star \leqslant e^{- 2 \mu t}(f(x_0) - f_\star).
\end{equation}
This convergence guarantee follows directly from the derivative with respect to time of the Lyapunov function $\mathcal{V}(X_t) = f(X_t) - f_\star$, together with strong convexity (or Łojasiewicz inequality): $\frac{d}{dt}\mathcal{V}(X_t) = \dot{X_t}^\top \nabla f(X_t) = -\|\nabla f(X_t)\|^2 \leqslant -2 \mu (f(X_t) - f_\star) = -2 \mu \mathcal{V}(X_t)$.

\bigbreak
Given the specific gradient flows studied in this work, it is reasonable to search for Lyapunov functions made of linear combinations of function values, and a quadratic form in the trajectory $X_t$. We simply refer to them as quadratic Lyapunov functions:
\begin{equation}
\label{lyapunov}
    \mathcal{V}_{a,c}(X_t) = a \cdot (f(X_t) - f_\star) + c \cdot \|X_t - x_\star\|^2,
\end{equation}
where $a, c \geqslant 0$ are constants that do not depend on $t$, such that the function $ \mathcal{V}_{a,c}$ is nonnegative and nonincreasing along the flow. Quadratic Lyapunov functions are common for proving convergence of gradient flows and cover for instance the Lyapunov function used to prove convergence of the gradient flow under strong convexity~\eqref{reflyapmu}. Given a Lyapunov function $\mathcal{V}_{a,c}$, the idea is to find the largest nonnegative value $\tau_\star$ such that the condition
\begin{equation}
\label{wc_mu}
    \frac{d}{dt}\mathcal{V}_{a,c}(X_t) \leqslant -\tau_\star \mathcal{V}_{a,c}(X_t),
\end{equation}
holds for any dimension $d \in \mathbf{N}$, any function $f \in \mathcal{F}_{\mu, \infty}$ and any trajectory $X_t$ generated by the gradient flow. After integration, \eqref{wc_mu} allows to obtain a convergence guarantee of the form $\mathcal{V}_{a,c}(X_t) \leqslant e^{-\tau_\star t} \mathcal{V}_{a,c}(X_0)$. Given a certain Lyapunov function $\mathcal{V}_{a,c}$ and a time $t$, we get that the largest acceptable $\tau_\star$ is a solution to:
\begin{equation}
\label{eq:str_pep}
    \begin{aligned}
      -\tau_\star = \max_{X_t \in \mathbf{R}^d, d \in \mathbf{N}, f\in \mathcal{F}_{\mu, \infty}} \ & \  \frac{d}{dt}\mathcal{V}_{a,c}(X_t), \\
      \text{subject to }
    \mathcal{V}_{a,c}(X_t) &= 1, \\
    \dot{X_t} &= - \nabla f (X_t).
\end{aligned}
\end{equation}
This minimization problem is invariant with respect to $t$. 
It is established in~\cite{taylor2016smooth, drori2012performance} that these so-called performance estimation problems (PEP) can be formulated as SDPs. Because of the variable $f \in \mathcal{F}_{\mu, \infty}$, the maximization problem~\eqref{eq:str_pep} is infinite-dimensional. Recall that a differentiable function $f \in \mathcal{F}_{\mu, \infty}$ verifies for all points $x,y \in \mathbf{R}^d$, $f(x) - f(y) - \langle \nabla f(y), x-y \rangle \geqslant \frac{\mu}{2} \|x-y\|^2$. Using alternate variables $f_t$, $f_\star$, $g_t$ and $g_\star$ (informally: $f_t = f(X_t)$, $f_\star = f(x_\star)$, $ g_t = \nabla f(X_t)$ and $g_\star = \nabla f(x_\star) = 0$), it holds that:
\begin{equation}
\label{eq:interpolation_step}
    \begin{aligned}
      -\tau_\star = \max_{X_t \in \mathbf{R}^d, d \in \mathbf{N}} \ & \ \frac{d}{dt}\mathcal{V}_{a, c}(X_t), \\
      \text{subject to } &\mathcal{V}_{a,c}(X_t) = 1, \\
    &\dot{X}_t = - g_t, \\
    & f_i - f_j -\langle g_j, X_i-X_j\rangle \geqslant \frac{\mu}{2}\|X_i-X_j\|^2, \ \forall i,j=t, \star.
\end{aligned}
\end{equation}
The fact that~\eqref{eq:interpolation_step} produces an upper bound on $-\tau_\star$ directly follows from the fact that any sampled strongly convex function satisfy these inequalities at the sampled points (here $X_t$ and $x_\star$). Thereby, any feasible point to~\eqref{eq:str_pep} corresponds to a feasible point for~\eqref{eq:interpolation_step} with the same objective value. In the other direction,~\cite[Corollary 2]{taylor2016smooth} (which provides a constructive way to obtain some $f \in \mathcal{F}_{\mu, \infty}$ that interpolates the triplets $(X_i, g_i, f_i)_{i=t, \star}$) ensures that any feasible point to~\eqref{eq:interpolation_step} can be translated to a feasible point to~\eqref{eq:str_pep} with also the same objective value, thereby reaching the equivalence between formulations~\eqref{eq:str_pep} and~\eqref{eq:interpolation_step}.
\smallbreak
In a second stage, we introduce $G  = \begin{pmatrix} \|X_t - x_\star\|^2 &  \langle X_t - x_\star, g_t \rangle  \\  \langle X_t-x_\star, g_t \rangle & \| g_t \|^2 \end{pmatrix} \succcurlyeq 0$, a Gram matrix and a vector $F = [f_t, f_\star]$, thereby obtaining a semidefinite reformulation:
\begin{equation}
\label{SDP}
    \begin{aligned}
      -\tau_\star = \max_{G\succcurlyeq 0, F \in \mathbf{R}^2} & b_0^\top F + {\rm Tr}(A_0G), \\
      \text{subject to }
    & b_1^\top F + {\rm Tr}(A_1G) \geqslant 0, \\
    & b_2^\top F + {\rm Tr}(A_2G) \geqslant 0, \\
    & b_3^\top F + {\rm Tr}(A_3G) = 1. \\
\end{aligned}
\end{equation}
where $A_0 = \begin{pmatrix} 0 & -c \\ -c & -a \end{pmatrix}$, $A_1 = \begin{pmatrix}-\mu/2 & 1/2 \\ 1/2 & 0 \end{pmatrix}$, $A_2= \begin{pmatrix}-\mu/2 & 0 \\ 0 & 0\end{pmatrix}$, $A_3 = \begin{pmatrix}c & 0 \\ 0 & 0 \end{pmatrix}$, $b_0 = [0, 0]^\top$, $b_1 = [-1,\ 1]^\top $, $b_2=[1, \ -1]^\top $ and $b_3 = a[1, \ -1]^\top $. Consider the corresponding Lagrangian, for $F \in \mathbf{R}^2$, $G \succcurlyeq 0$, $\tau \in \mathbf{R}$ and $\lambda_1, \lambda_2 \geqslant 0$: 
\begin{equation*}
    \begin{aligned}
    \mathcal{L}(F, G, \tau, \lambda_1, \lambda_2) =& b_0^\top F + {\rm Tr}(A_0G) + \tau \cdot (b_3^\top F + {\rm Tr}(A_3G) - 1) \\ 
    & + \lambda_1 \cdot(b_1^\top F + {\rm Tr}(A_1G)) + \lambda_2 \cdot(b_2^\top F + {\rm Tr}(A_2G)).
    \end{aligned}
\end{equation*}
The saddle point of the Lagrangian is given by:
\begin{equation*}
    \begin{aligned}
    \tau_\star &= \min_{\tau, \lambda_1 \geqslant 0, \lambda_2 \geqslant 0} \max_{F, G \succcurlyeq 0} \mathcal{L}(F, G, \tau, \lambda_1, \lambda_2).
\end{aligned}
\end{equation*}
The Lagrangian dual of the SDP~\eqref{SDP} is obtained by maximizing over $F \in \mathbf{R}^2$, $G \succcurlyeq 0$:
\begin{equation*}
    \begin{aligned}
    -\tau_\star =  \min_{\tau, \lambda_1, \lambda_2} & \ -\tau, \\
      \text{subject to } & S = A_0 + \lambda_1 A_1 + \lambda_2 A_2 + \tau A_3 \preccurlyeq 0 \\ 
    &  b_0 +  \lambda_1 b_1 +\lambda_2 b_2  + \tau b_3 = 0, \\
    &  \tau\in \mathbf{R}, \ \lambda_1, \lambda_2 \geqslant 0.
\end{aligned}
\end{equation*}
The equality with $\tau_\star$ comes from strong duality, that holds via Slater's conditions (it is relatively easy to show that there exists a Slater's point using the same construction as in~\cite[Theorem 6]{taylor2016smooth}). Finally, since any feasible $\tau$ is a lower bound for $\tau_\star$, those developments allow arriving to the equivalence between verifying a quadratic Lyapunov function and verifying the feasibility of an LMI.

\begin{theorem}
\label{conv_gf_mu}
Let $a,c, \tau \geqslant 0$ and $\mu > 0$. The following assertions are equivalent:
\begin{itemize}
    \item The inequality $\frac{d}{dt}\mathcal{V}_{a,c}(X_t) \leqslant -\tau \mathcal{V}_{a,c}(X_t)$, is satisfied for all dimension $d \in \mathbf{N}$, for all $f \in \mathcal{F}_{\mu, \infty}$ and all trajectory $X_t$ solutions to the gradient flow~\eqref{eq:gf_init}, where $\mathcal{V}_{a,c}$ is a quadratic Lyapunov function~\eqref{lyapunov}.
    \item There exist $\lambda_1, \lambda_2 \geqslant 0$ such that:
    \begin{equation}
\label{LMI}
S = \begin{pmatrix} \tau c - \frac{\mu}{2}(\lambda_1 + \lambda_2) & -c + \frac{\lambda_1}{2} \\ -c + \frac{\lambda_1}{2} & -a \end{pmatrix} \preccurlyeq 0, \ \tau a = \lambda_1 - \lambda_2.
\end{equation}
\end{itemize}
\end{theorem}

\begin{remark}
As a corollary of our result and for the class of quadratic Lyapunov functions~\eqref{lyapunov} (and later~\eqref{lyap_conv},~\eqref{lyapunov_order2}), it turns out only two interpolation inequalities in $(X_t, x_\star)$ and $(x_\star, X_t)$ are involved in convergence proofs for continuous-time models (see Theorem~\ref{conv_theorem_order2} for second-order gradient flows). In other words, the framework reveals shorter proofs in continuous-time.
\end{remark}

A few conclusions can be drawn from the LMI equivalence from Theorem~\ref{conv_gf_mu}. For fixed value of $a, c, \tau$, the LMI provides provides a necessary and sufficient condition for a quadratic Lyapunov function $\mathcal{V}_{a,c}$ to decrease at a specific rate $\tau \geqslant 0$ for all function in the class $ \mathcal{F}_{\mu, \infty}$. Second, we can simultaneously optimize over the class of quadratic Lyapunov functions and over the convergence rate. Indeed, given a rate~$\tau$, the LMI is jointly convex in $\lambda_1, \lambda_2, a, c$. Therefore, thanks to linearity of the feasibility problem in $\tau$, a bisection search allows to optimize over  it and to find the worst-case guarantee~$\tau_\star$.

\begin{figure}[!ht]
\begin{subfigure}[t]{0.47\textwidth}
  \centering
  \begin{tikzpicture}
\begin{loglogaxis}[legend pos=north west,legend style={draw=none},plotOptions1,width=1.0\linewidth,ylabel={Worst-case convergence rate $\tau$},xlabel={Strong convexity parameter $\mu$}]
\addplot[blue, thick] table [y=theorygf, x=condition]{gf.txt}; 
\addplot[dotted, red, very thick] table [y=gf, x=condition]{gf.txt};
\addlegendentry{$2 \mu$}
\addlegendentry{PEP}
\end{loglogaxis}
\end{tikzpicture}
  \caption{Worst-case rate $\tau_\star$ for the class of quadratic Lyapunov functions~\eqref{lyapunov}}.
  \label{fig:gf_rate}
\end{subfigure}
\hspace{.2cm}
\begin{subfigure}[t]{0.47\textwidth}
  \centering
  \begin{tikzpicture}
\begin{axis}[legend pos=north east,legend style={draw=none},plotOptions2,width=1.0\linewidth,ylabel={$f(X_t) - f_\star$},xlabel={$X_t - x_\star$}]
\addplot[blue, thick] table [y=functheorygf, x=points]{worst_function_gf.txt};
\addplot[red, dotted, thick] table [y=funcgf, x=points]{worst_function_gf.txt};
\addlegendentry{$\frac{\mu}{2}x^2$}
\addlegendentry{PEP}
\end{axis}
\end{tikzpicture}
  \caption{Reconstruction of a function $f \in \mathcal{F}_{\mu,\infty}$ that interpolates $x_\star$ and $X_t$, while matching the convergence rate $\tau_\star = 2 \mu$, with $\mu=0.1$.}
  \label{fig:GF_function}
  \end{subfigure}
\vspace{-1.em}
 \caption{Comparison between numerical values for $\tau$ obtained by solving the LMI~\eqref{LMI} and the reference established in the literature~\cite[Proposition 1.1]{scieur2017integration}, for trajectories~$X_t$ generated by gradient flow~\eqref{eq:gf_init} originating from a $\mu$-strongly convex function.}
\label{fig:examplesgf}
\vspace{-1.5em}
\end{figure}
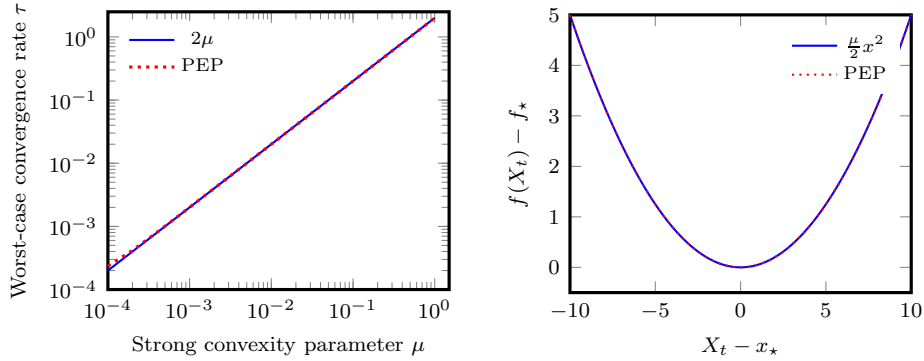


In Figure~\ref{fig:gf_rate}, we obtain the fastest linear convergence rate that can be achieved using quadratic Lyapunov functions~\eqref{lyapunov} (and even for all Lyapunov functions, since the rate is tight on a function $f(x) = \frac{\mu}{2}\|x\|_2^2$ for any time $t$). Together with Theorem~\ref{conv_gf_mu} we retrieve the known linear worst-case convergence speed in $e^{-2\mu}$ from Scieur et~al.~\cite[Proposition 1.1]{scieur2017integration} without improvement. The numerical approach allows to ensure tightness of the procedure by construction, as it guarantees the existence of a numerical function $f$ that exactly achieves this convergence guarantee (see Figure~\ref{fig:GF_function} and the method in~\cite[Chapter 3]{taylor2017thesis}). The next section builds on the same technique to analyze the gradient flow originating from a (possibly non strongly) convex function. In this scenario, the difficulty comes from the time-dependence of Lyapunov functions.

\subsubsection{Minimizing convex functions}

Let $f \in \mathcal{F}_{0, \infty}$ and $x_\star$ be a minimizer of $f$. In this case, worst-case convergence rates are often sublinear. Again, as in discrete-time, it is possible to obtain convergence guarantees using time-dependent quadratic Lyapunov functions. 
\smallbreak
The Lyapunov function $\mathcal{V}(X_t,t) = t(f(X_t) - f_\star) + \frac{1}{2}\|X_t - x_\star\|^2$ from~\cite[p. 7]{Su2015} verifies $\frac{d}{dt}\mathcal{V}(X_t,t) \leqslant 0$ for any dimension $d \in \mathbf{N}$, any function $f \in \mathcal{F}_{0, \infty}$, and any trajectory $X_t$ generated by the gradient flow~\eqref{eq:gf_init} (\textit{proof}: $\frac{d}{dt}\mathcal{V}(X_t) = t\langle \nabla f(X_t), \dot{X}_t\rangle + f(X_t) - f_\star + \langle \dot{X}_t, X_t - x_\star\rangle = -  t\| \nabla f(X_t)\|^2 + f(X_t) - f_\star - \langle \nabla f(X_t), X_t - x_\star\rangle \leqslant -  t\| \nabla f(X_t)\|^2$ using convexity). After integration between $0$ and $t$, we recover a convergence bound in function values from the literature~\cite[p.7]{Su2015},~\cite[Section 6.3.1]{fazlyab2018analysis}:
\begin{equation*}
    f(X_t) - f_\star \leqslant \frac{\|x_0 - x_\star\|^2}{2t}.
\end{equation*}
\smallbreak
Let us adapt the techniques from Section~\ref{sec:methodo} by considering quadratic Lyapunov functions:
\begin{equation}
\label{lyap_conv}
    \mathcal{V}_{a_t, c_t}(X_t,t) = a_t (f(X_t) - f_\star) + c_t \|X_t - x_\star\|^2,
\end{equation}
where $a_t, c_t \geqslant 0$ are functions differentiable with respect to time such that the function $\mathcal{V}_{a_t, c_t}$ is nonnegative and nonincreasing along the flow. When a quadratic Lyapunov function decreases along the trajectory $X_t$, that is $\frac{d}{dt}\mathcal{V}_{a_t, c_t}(X_t) \leqslant 0$, a convergence guarantee in function values is given by
\begin{equation*}
    f(X_t) - f_\star \leqslant \frac{\mathcal{V}_{a_0, c_0}(x_0,0)}{a_t}=\frac{a_0(f(x_0) - f_\star) + c_0\|x_0 - x_\star\|^2}{a_t}.
\end{equation*}
Verifying a quadratic Lyapunov function can be cast as verifying that the following maximization problem is nonpositive:
\begin{equation*}
    \begin{aligned}
      0 \geqslant \max_{X_t \in \mathbf{R}^d, d\in \mathbf{N},f\in \mathcal{F}_{0, \infty}} \ & \ \frac{d}{dt} \mathcal{V}_{a_t, c_t}\\
      \text{subject to }
    &\ \dot{X}_t = - \nabla f (X_t).
\end{aligned}
\end{equation*}

\begin{remark}
\label{remark_lyap}The strongly convex case as defined above is a particular case of the convex one, using a specific Lyapunov function $\Phi(\cdot)$, such that $\mathcal{V}(X_t,t) = e^{\tau t}\Phi(X_t)$ where $\Phi(X_t) = a\cdot (f(X_t) - f_\star) + c\cdot \|X_t - x_\star\|^2$. Then, $\frac{d}{dt}\mathcal{V}(X_t,t) \leqslant 0$ is equivalent to $\frac{d}{dt}\Phi(X_t) \leqslant -\tau \Phi(X_t)$.
\end{remark}
\smallbreak

\begin{theorem}
\label{conv_gf_conv} 
Let $a_t, c_t \geqslant 0$ be continuously differentiable with respect to time. The following assertions are equivalent:
\begin{itemize}
    \item The inequality
    $    \displaystyle \frac{d}{dt}\mathcal{V}_{a_t, c_t}(X_t,t) \leqslant 0
   $, is satisfied for all dimension $d \in \mathbf{N}$, all function $f \in \mathcal{F}_{0, \infty}$ and all trajectory $X_t$ generated by the gradient flow~\eqref{eq:gf_init}, where $\mathcal{V}_{a_t, c_t}$ is a quadratic Lyapunov function defined in~\eqref{lyap_conv}.
    \item There exist $\lambda^{(1)}_t, \lambda^{(2)}_t \geqslant 0$ such that:
\begin{equation*}
S = \begin{pmatrix} \dot{c}_t & -c_t + \frac{\lambda^{(1)}_t}{2} \\ -c_t + \frac{\lambda^{(1)}_t}{2} & -a_t \end{pmatrix} \preccurlyeq 0, \ \dot{a}_t = \lambda^{(1)}_t - \lambda^{(2)}_t.
\end{equation*}
\end{itemize}
\end{theorem}
\begin{proof}
The LMI is obtained following the previous methodology (Appendix~\ref{sec:A0_conv}).
\end{proof}
Choosing $\lambda^{(1)}_t=1$, $\lambda^{(2)}_t=0$, together with $c_t = \frac{1}{2}$ and $a_t = t$, the conditions in the LMI from Theorem~\ref{conv_gf_conv} are satisfied. We retrieve the Lyapunov function $\mathcal{V}(x,t) = t(f(x) - f_\star) + \frac{1}{2}\|x - x_\star\|^2$ from~\cite[p. 7]{Su2015}.

\smallbreak
Similar to Theorem~\ref{conv_gf_mu}, the LMI from Theorem~\ref{conv_gf_conv}, and hence the problem of looking for a Lyapunov function, is jointly convex in $\lambda^{(1)}_t$, $\lambda^{(2)}_t$, $c_t$, $a_t$, $\dot{a}_t$, $\dot{c}_t$. This Lyapunov analysis can also be validated numerically, as for the gradient flow originating from strongly convex functions.

\subsection{Accelerated gradient flows}

 A major improvement to gradient descent dates back to Nesterov~\cite{nesterov1983}, who introduced an accelerated gradient method (AGM): 
\begin{equation}
\begin{aligned}
    x_{k+1} &= y_{k} - \gamma \nabla f(y_{k}), \\
    y_{k+1} &= x_{k+1} + \alpha_k(x_{k+1} - x_{k}),
\end{aligned}
\label{eq:agm}
\end{equation}
where $\gamma, \alpha_k \geqslant 0$ depend on the class of functions to minimize. The combination of past iterates allows more control over the accumulated error. The idea of incorporating momentum was first introduced by Polyak~\cite{polyak1964} with the heavy-ball method, starting from~$x_0, x_1 \in \mathbf{R}^d$, and for a momentum $\alpha_k > 0$:
\begin{equation}
    x_{k+2} = x_{k+1} + \alpha_k(x_{k+1} - x_k) - \gamma \nabla f(x_{k+1}).
\label{eq:hbm}
\end{equation}
 Yet, compared to Nesterov's accelerated gradient method, the heavy-ball method lacks global acceleration beyond quadratics.
\smallbreak
When the step size $\gamma$ goes to zero, these schemes happen to be closely related to second-order differential equations, where $\beta_t \geqslant 0$ is a continuous function depending on $\alpha_k$:
\begin{equation*}
    \ddot{X}_t + \beta_t \dot{X} + \nabla f(X_t) = 0.
\end{equation*}
Recently, accelerated gradient methods have been analyzed using second-order differential equations~\cite{shi2018understanding, sanzserna2021connections, wilson2016lyapunov, hu2017dissipativity}. Reversely, the accelerated gradient method and the heavy-ball method may be seen as discretization schemes of these second-order ODEs, as many other schemes. Discretization techniques are, among others, discussed by~\cite{xu2018, shi2020learning, wilson2016lyapunov}. Taking integration theory's point of view, Scieur et~al.~\cite{scieur2017integration} proved that these multi-step methods may even be seen as discretization schemes of the gradient flow (for quadratics).
\smallbreak
Again, ODEs and multi-step first-order methods as defined above are often handled using quadratic Lyapunov functions. We extend the systematic Lyapunov approach developed previously to accelerated gradient flows. Let $\mathcal{V}_{a_t, P_t}$ be a quadratic Lyapunov function for second-order gradient flows, taken in the class of quadratic functions:
\begin{equation}
\label{lyapunov_order2}
    \mathcal{V}_{a_t, P_t}(X_t, t) = a_t( f(X_t) - f_\star) + \begin{pmatrix} X_t - X_\star\\ \dot{X}_t \end{pmatrix}^\top  \!\!\! \left(P_t \otimes I_d\right) \begin{pmatrix} X_t - X_\star\\ \dot{X}_t \end{pmatrix},
\end{equation}
where $P = \begin{pmatrix}p^{(11)}_t &  p^{(12)}_t \\ p^{(12)}_t &  p^{(22)}_t \end{pmatrix}, a_t \geqslant 0$ are continuously differentiable with respect to time and such that the Lyapunov function is nonnegative and nonincreasing along the flow. After integration of a quadratic Lyapunov function $\mathcal{V}_{a_t, c_t}$ between $0$ and $t$, this approach leads to convergence bounds for instance in function values $f(X_t) - f_\star \leqslant \frac{\mathcal{V}(x_0)}{a_t}$. 
\smallbreak

\subsubsection{Minimizing strongly convex functions}
Let $f \in \mathcal{F}_{\mu, L}$, with strong convexity parameter $\mu > 0$, and let Nesterov's accelerated gradient methods' parameters by defined by $\gamma = \frac{1}{L}$ and $\alpha= \frac{1 - \sqrt{\mu \gamma}}{1 + \sqrt{\mu \gamma}}$ in~\eqref{eq:agm}. When the step size $\gamma$ goes to zero, the continuous-time limit of $y_k$ in~\eqref{eq:agm} is exactly the Polyak damped oscillator~\cite{polyak1964}:
\begin{equation} 
\label{str_agf}
    \ddot{X_t} + 2\sqrt{\mu}\dot{X_t} + \nabla f(X_t) = 0,
\end{equation} 
as it has already been highlighted in previous works~\cite{fazlyab2018analysis, sanzserna2021connections, shi2018understanding}. The heavy-ball method~\eqref{eq:hbm} reaches the same continuous-time limit when reducing the step size. Shi et~al.~\cite{shi2018understanding} derived a convergence guarantee in $f(X_t) - f_\star = \mathcal{O}(e^{-\frac{\sqrt{\mu}t}{4}})$ using a Lyapunov-based approach. This bound was improved to $f(X_t) - f_\star = \mathcal{O}(e^{-\sqrt{\mu}t})$, by Wilson~et~al.~\cite[Appendix B]{wilson2016lyapunov} using the Bregman-Lagrangian approach, and by Sanz-Serna~and~Zygalakis using the IQC framework~\cite{sanzserna2021connections, fazlyab2018analysis}. Using the methodology from Theorem~\ref{conv_gf_mu}, we show that verifying linear convergence guarantees using quadratic Lyapunov functions with constant parameters~\eqref{lyapunov_order2} can be cast as an LMI.

\begin{theorem}
\label{impr_conv_agf_mu}
Let $\mu>0$ and $\tau \geqslant 0$. Let $a \geqslant 0$, and $P$ be a symmetric matrix. The following assertions are equivalent:
\begin{itemize}
    \item The inequality $\displaystyle\frac{d}{dt}\mathcal{V}_{a, P}(X_t) \leqslant -\tau \mathcal{V}_{a, P}(X_t)$, is satisfied for all dimension $d\in \mathbf{N}$, all function $f \in \mathcal{F}_{\mu, \infty}$ and all trajectory $X_t$ generated by the Polyak damped oscillator~\eqref{str_agf}, where $\mathcal{V}_{a, P}$ is a quadratic Lyapunov function~\eqref{lyapunov_order2}.
\item There exist $\lambda_1, \lambda_2, \nu_1, \nu_2 \geqslant 0$, such that:
\begin{equation}
\begin{aligned}
\label{LMI_acc}
 & \begin{pmatrix} -\frac{\mu}{2}(\lambda_1 + \lambda_2) + \tau p_{11} & p_{11} - 2\sqrt{\mu}p_{12} + \tau p_{12} & - p_{12} + \frac{\lambda_1}{2} \\ p_{11} - 2\sqrt{\mu}p_{12} + \tau p_{12} & 2(p_{12} - 2\sqrt{\mu}p_{22}) + \tau p_{22} & -p_{22} + \frac{a}{2} \\ - p_{12} + \frac{\lambda_1}{2} & -p_{22} + \frac{a}{2} & 0\end{pmatrix} \preccurlyeq  0, \\
        & \tau a = \lambda_1 - \lambda_2, \\
        & \begin{pmatrix} P & 0 \\ 0 & 0 \end{pmatrix} + \begin{pmatrix} \frac{\mu}{2}(\nu_1+\nu_2) & 0 & \frac{-\nu_1}{2} \\ 0 & 0 & 0 \\ \frac{-\nu_1}{2} & 0 & 0\end{pmatrix} \succcurlyeq 0 , \\
        & a = \nu_2 - \nu_1.
\end{aligned}
\end{equation}
\end{itemize}
\end{theorem}
\begin{proof}
The equivalence is obtained using the methodology developed in Section~\ref{sec:methodo} and introducing the Gram matrix $G = P^\top P$, where $P = (\dot{X}_t, X_t - x_\star, g_t)$, where $g_t$ holds for $\nabla f(X_t)$. The first LMI refers to the non-increasing condition for the Lyapunov function. The second one refers to the positivity constraint $\mathcal{V}_{a, P}(X_t) \geqslant 0 $, for all dimension $d\in \mathbf{N}$, all function $f \in \mathcal{F}_{0, \infty}$, and all trajectory $X_t$ generated by Polyak damped oscillator~\eqref{str_agf}, as done for discrete-time methods in~\cite[Th. 7]{taylor2018lyapunov}.
\end{proof}

As for the gradient flow, the LMI is jointly convex in $\lambda_1, \lambda_2, \nu_1, \nu_2 \geqslant 0$, in  the Lyapunov parameters $a, P$, and is linear in $\tau$. Hence, we can perform a bisection search over $\tau$ to find the fastest linear convergence rate that can be verified using quadratic Lyapunov functions, as done in Figure~\ref{fig:agf_rate}. The framework provides in addition a numerical tool for choosing Lyapunov parameters for which the worst-case linear convergence rate is achieved. Figure~\ref{fig:AGF_lyapunov} helped providing an intuition for parameters in Corollary~\ref{coro:str_agf}.

\begin{corollary}
\label{coro:str_agf}
Let $\mu \geqslant 0$. The function
\begin{equation*}
    \mathcal{V}(X_t) = f(X_t) - f_\star +  \begin{pmatrix} X_t  - X_\star\\ \dot{X_t} \end{pmatrix}^\top  \!\!\! \left(\begin{pmatrix} {4}/{9}\mu & {2}/{3}\sqrt{\mu} \\  {2}/{3}\sqrt{\mu} & {1}/{2} \end{pmatrix}\otimes I_d\right)\!\begin{pmatrix} X_t - X_\star\\ \dot{X_t} \end{pmatrix},
\end{equation*}
verifies $\frac{d}{dt}\mathcal{V}(X_t) \leqslant -{4}/{3}\sqrt{\mu} \mathcal{V}(X_t)$ for all dimension $d \in \mathbf{N}$, all function $f \in \mathcal{F}_{\mu, \infty}$, and all trajectory $X_t$ generated by the Polyak damped oscillator~\eqref{str_agf}. Tight rate is achieved for $f(x) = \frac{1}{2}\mu x^2$.
\end{corollary}
\begin{proof}
Taking $\lambda_1 = {4}/{3}\sqrt{\mu}$, $\lambda_2 = 0$, $\nu_1 = 0$ and $\nu_2 = 1$, we verify the LMI for this Lyapunov function $\mathcal{V}$, with $\tau = {4}/{3}\sqrt{\mu}$.
\end{proof}

This class of quadratic Lyapunov functions is inspired by~\cite{taylor2018lyapunov} in discrete-time, where a stricter positivity condition on $P \succcurlyeq 0$ hindered proving tight convergence of Nesterov's accelerated gradient. Similarly in our context, the Lyapunov function from Corollary~\ref{str_agf} is defined by $P=\begin{pmatrix} {4}/{9}\mu & {2}/{3}\sqrt{\mu} \\  {2}/{3}\sqrt{\mu} & {1}/{2} \end{pmatrix}$, which is not positive semidefinite. Usually in the continuous-time models' literature, we only consider matrices $P$ that are positive semidefinite, such as in the Lyapunov function from~\cite[Theorem 4.3]{sanzserna2021connections, shi2018understanding}: 
\begin{equation*}
    \mathcal{V}(X_t) = f(X_t) - f_\star + \frac{1}{2} \begin{pmatrix} X_t - X_\star\\ \dot{X_t} \end{pmatrix}^\top  \!\!\! \left(\begin{pmatrix} \mu & \sqrt{\mu} \\ \sqrt{\mu} & 1 \end{pmatrix}\otimes I_d\right)\!\begin{pmatrix} X_t - X_\star\\ \dot{X_t} \end{pmatrix}
\end{equation*}
that verifies $\displaystyle\frac{d}{dt}\mathcal{V}(X_t) \leqslant -\sqrt{\mu} \mathcal{V}(X_t)$ for all dimension $d \in \mathbf{N}$, all function $f \in \mathcal{F}_{\mu, \infty}$ and all trajectory $X_t$ generated by Polyak damped oscillator. This Lyapunov is a feasible point of the LMI~\eqref{LMI_acc} from Theorem~\ref{impr_conv_agf_mu}, with $\tau = \sqrt{\mu}$, $\lambda_1=\sqrt{\mu}$, $\lambda_2=0$, $\nu_1=0$, $\nu_2=a=1$. By relaxing the condition $P \succcurlyeq 0$, we thus improve results from Sanz-Serna~and~Zygalakis and Wilson~et~al. by a factor $4/3$ in Corollary~\ref{coro:str_agf}.

\begin{figure}[!ht]
\begin{subfigure}[t]{0.47\textwidth}
  \centering
  \begin{tikzpicture}
\begin{loglogaxis}[legend pos=south east,legend style={draw=none},plotOptions3,width=1.0\linewidth,ylabel={Worst-case $\tau$},xlabel={Strong convexity parameter $\mu$}]
\addplot[blue, thick] table [y=theoryagf, x=condition]{agf.txt}; 
\addplot[dotted, blue, very thick] table [y=agf, x=condition]{agf.txt};
\addplot[red, thick] table [y=theoryagf, x=condition]{relaxed_agf.txt}; 
\addplot[dotted, red, very thick] table [y=agf, x=condition]{relaxed_agf.txt};
\addlegendentry{$\sqrt{\mu}$ ($P \succcurlyeq 0$)}
\addlegendentry{PEP ($P \succcurlyeq 0$)}
\addlegendentry{$4/3\sqrt{\mu}$}
\addlegendentry{PEP}
\end{loglogaxis}
\end{tikzpicture}
  \caption{Best guarantees found within the class of quadratic Lyapunov functions~\eqref{lyapunov_order2}.}
  \label{fig:agf_rate}
\end{subfigure}
\hspace{0.5cm}
\begin{subfigure}[t]{0.47\textwidth}
  \centering
  \begin{tikzpicture}
\begin{loglogaxis}[legend pos=south east,legend style={draw=none},plotOptions4,width=1.0\linewidth,ylabel={Lyapunov parameters},xlabel={Strong convexity parameter $\mu$}]
\addplot[green, thick] table [y=p11, x=condition]{P_agf.txt};
\addplot[orange, thick] table [y=p01, x=condition]{P_agf.txt};
\addplot[orange, dotted, very thick] table [y=ref1, x=condition]{P_agf.txt};
\addplot[purple, thick] table [y=p00, x=condition]{P_agf.txt};
\addplot[purple, dotted, very thick] table [y=ref2, x=condition]{P_agf.txt};
\addlegendentry{PEP $p_{22}$}
\addlegendentry{PEP $p_{12}$}
\addlegendentry{$4/9\mu$}
\addlegendentry{PEP $p_{11}$}
\addlegendentry{$2/3\sqrt{\mu}$}
\end{loglogaxis}
\end{tikzpicture}
  \caption{Lyapunov parameters $P$ in~\eqref{lyapunov_order2} for $\tau = 4/3\sqrt{\mu}$ and $a=1$, as a function of the condition number $\mu$.}
  \label{fig:AGF_lyapunov}
  \end{subfigure}
\vspace{-1.5em}
 \caption{Comparison between the worst-case guarantee obtained numerically with PEP, and its references, for the Polyak damped oscillator~\eqref{str_agf} originating from $\mu$-strongly convex functions, and for quadratic Lyapunov functions~\eqref{lyapunov_order2}.}
\label{fig:examplesagf}
\vspace{-1.5em}
\end{figure}
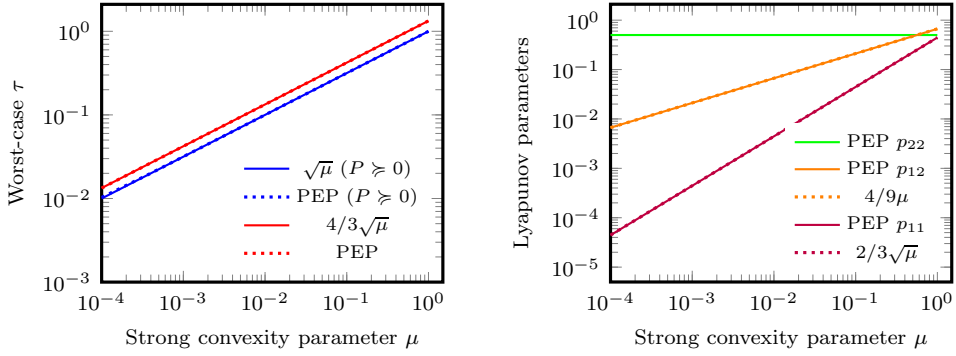
 Figure~\ref{fig:agf_rate} together with Corollary~\ref{coro:str_agf} allows to conclude that this bound cannot be improved when changing the Lyapunov function among the class of quadratic functions~\eqref{lyapunov_order2}.

\subsubsection{Minimizing convex functions}
As for the gradient flow, rates are sublinear when the accelerated gradient flow originates from convex functions. Let $f\in \mathcal{F}_{0, L}$, $\gamma \leqslant \frac{1}{L}$ be the step size, and $\alpha_k = \frac{k-1}{k+2}$ be the scheme parameter in Nesterov's accelerated method. Su et~al.~\cite[Section 2]{Su2015} proved the connection between the first-order scheme and a second-order ODE known as the accelerated gradient flow~(AGF):
\begin{equation}
    \ddot{X_t} + \frac{3}{t}\dot{X_t} + \nabla f(X_t) = 0.
\label{eq:nest_flow}
\end{equation}
Su et~al.~\cite[Theorem 3]{Su2015} proved that the following inequality is verified for all function $f\in\mathcal{F}_{0,\infty}$ and all trajectory $X_t$ generated by the accelerated gradient flow~\eqref{eq:nest_flow}:
\begin{equation*}
    f(X_t) - f_\star \leq 2 \frac{\|x_0 - x_\star\|^2}{t^2}.
\end{equation*}
Their proof exhibits a Lyapunov function $\mathcal{V}(X_t,t) = t^2(f(X_t) - f_\star) + 2\|(X_t-x_\star) + \frac{t}{2}\dot{X}_t\|^2$, which is decreasing along trajectories $X_t$ (\textit{proof: }$\frac{d}{dt}\mathcal{V}(X_t, t) = 2t(f(X_t) - f_\star) + t^2\langle \dot{X}_t, \nabla f(X_t)\rangle + 2\langle X_t - x_\star + \frac{t}{2}\dot{X}_t, 3\dot{X}_t + t\ddot{X}_t \rangle = 2t(f(X_t) - f_\star) - \langle \nabla f(X_t), X_t - x_\star\rangle \leqslant 0$ by convexity of $f$). In this context, we obtain again an LMI equivalence for verifying a quadratic Lyapunov function~\eqref{lyapunov_order2}.

\begin{theorem}
\label{conv_theorem_order2}
Let $a_t \geqslant 0$, $P_t \succcurlyeq 0$ be functions continuously differentiable with respect to time. The following assertions are equivalent:
\begin{itemize}
    \item The inequality $\displaystyle\frac{d}{dt}\mathcal{V}_{a_t, P_t}(X_t,t) \leqslant 0$, is satisfied for all dimension $d \in \mathbf{N}$, all function $f \in \mathcal{F}_{0, \infty}$ and all trajectory $X_t$ generated by the accelerated gradient flow~\eqref{eq:nest_flow}, where $\mathcal{V}_{a_t, P_t}$ is a quadratic Lyapunov function~\eqref{lyapunov_order2}.
    \item There exist $\lambda^{(1)}_t, \lambda^{(2)}_t \geqslant 0$ such that:
    \begin{equation*}
    \begin{aligned}
 S &= \begin{pmatrix} \dot{p}^{(11)}_t & p^{(11)}_t - \frac{3}{t}p^{(12)}_t + \dot{p}^{(12)}_t & - p^{(12)}_t + \frac{\lambda^{(1)}_t}{2} \\ p^{(11)}_t - \frac{3}{t}p^{(12)}_t + \dot{p}^{(12)}_t & 2(p^{(12)}_t - \frac{3}{t}p^{(22)}_t) + \dot{p}^{(22)}_t & -p^{(22)}_t + \frac{a_t}{2} \\ - p^{(12)}_t + \frac{\lambda^{(1)}_t}{2} & -p^{(22)}_t + \frac{a_t}{2} & 0\end{pmatrix} \preccurlyeq 0, \\
\dot{a}_t &= \lambda^{(1)}_t - \lambda^{(2)}_t.
    \end{aligned}
\end{equation*}
\end{itemize}
\end{theorem}
\begin{proof}
The proof follows those from Theorem~\ref{impr_conv_agf_mu}~and~\ref{conv_gf_mu}.
\end{proof}

 The Lyapunov function $\mathcal{V}(X_t,t) = t^2(f(X_t) - f_\star) + 2 \big\|(X_t-x_\star) + \frac{t}{2}\dot{X}\big\|^2$ exhibited by Su et al.~\cite[Theorem 3]{Su2015} is a feasible point of the LMI, with $a_t = t^2$ and $P_t = 2 \begin{pmatrix} 1 & t/2 \\ t/2 & t^2/4 \end{pmatrix}$ for the Lyapunov parameters, and $\lambda^{(1)}_t=t$, $\lambda^{(2)}_t=0$. It is possible to retrieve these results numerically, as it was done for the accelerated gradient flow originating from strongly convex functions.

\subsection{Higher-order convergence and time dilation}

In this section, we analyze convergence of the first and second-order non-autonomous gradient flows, and provide convergence guarantees depending on their parametrization. It appears that higher-order convergence of non-autonomous gradient flows is highly connected to time dilation. 

\subsubsection{A non-autonomous first-order gradient flow}

Let $f \in \mathcal{F}_{0, \infty}$. Let the non-autonomous first-order gradient flow be defined by:
\begin{equation}
    \label{alpha_gf}
    \dot{X}_t = - \alpha_t \nabla f (X_t), \ X_0 = x_0 \in \mathbf{R}^d,
\end{equation}
where $\alpha_t \geqslant 0$ is a continuous function (such that the flow is converging). It is natural to wonder if it is possible to accelerate such gradient flows by changing $\alpha_t$. A change of variable connects this ODE to the gradient flow~\eqref{eq:gf_init}, for which $\alpha_t=1$. Let $Y_t$ be the solution to the gradient flow, and $\tau_t = \int_0^t \alpha_sds$ be a time rescaling. Then, the variable  $X_t = Y_{\tau_t}$ verifies $\dot{X}_t = \frac{d}{dt}Y_{\tau_t} = \alpha_t \dot{Y}_{\tau_t} = -\alpha_t \nabla f(Y_{\tau_t}) = -\alpha_t \nabla f(X_t)$, which is exactly the non-autonomous gradient flow. The following corollaries can be obtained by performing this change of variable in Theorem~\ref{conv_gf_conv}.



\begin{corollary}
\label{alpha_coro_1}
Let $\mu \geqslant 0$.
\begin{itemize}
    \item If $\mu > 0$, the function $\mathcal{V}(X_t,t) = e^{2\mu\int_0^t  \alpha_s ds}(f(X_t)-f_\star)$ verifies $\displaystyle\frac{d}{dt}\mathcal{V}(X_t,t) \leqslant -2 \mu \alpha_t \mathcal{V}(X_t, t)$ for all dimension $d \in \mathbf{N}$, all function $f \in \mathcal{F}_{\mu, \infty}$ and all trajectory $X_t$ generated by the non-autonomous gradient flow~\eqref{alpha_gf}.
        \smallbreak
        A convergence guarantee is given by $\displaystyle f(X_t) - f_\star \leqslant e^{-2\mu\int_0^t  \alpha_s ds}(f(x_0) - f_\star)$.
    \item If $\mu=0$, the function $\mathcal{V}(X_t,t) = \left(\int_{0}^t\alpha_s ds\right)(f(X_t)-f_\star) + \frac{1}{2}\|X_t - x_\star\|^2$ verifies $\displaystyle\frac{d}{dt}\mathcal{V}(X_t,t) \leqslant 0$ for all dimension $d\in \mathbf{N}$, all function $f \in \mathcal{F}_{0, \infty}$ and all trajectory $X_t$ generated by the non-autonomous gradient flow~\eqref{alpha_gf}. 
    \smallbreak
    A convergence guarantee is given by $\displaystyle f(X_t) - f_\star \leqslant \frac{1}{2\int_0^t \alpha_s  ds}\|x_0 - x_\star\|^2$.
\end{itemize}
\end{corollary}

\begin{remark}
When $\alpha_t = 1$ above, that is $\tau_t=t$, we recover exactly the Lyapunov functions $\mathcal{V}(X_t,t) = t(f(X_t) - f_\star) + \frac{1}{2}\|X_t - x_\star \|^2$ from Theorem~\ref{conv_gf_conv} for convex functions, and $\mathcal{V}(X_t) = e^{2\mu t}(f(X_t) - f_\star)$ from Theorem~\ref{conv_gf_mu} for strongly convex functions.
\end{remark}

As mentioned by Orvieto and Lucchi~\cite{orvieto2020continuoustime} in the stochastic setting, and for accelerated methods by Wibisono et~al.~\cite{wibisono2016variational}, one can thus either work with $X_t$ generated by the non-autonomous gradient flow~\eqref{alpha_gf}, or with $Y_t$ generated by the gradient flow. However, acceleration on $X_t$ is not preserved after discretizing the flow. For example, applying explicit Euler scheme to the non-autonomous gradient flow~\eqref{alpha_gf} with $\alpha_s = \alpha > 0$ and originating from a function $f\in \mathcal{F}_{0, L}$, a condition on step sizes $h$ arises $ 0 \leqslant h \leqslant \frac{2}{L\alpha}$. 
 \smallbreak
When focusing on continuous-time models for analyzing explicit first-order methods, we usually prefer working with the gradient flow~\eqref{eq:gf_init}. However, the non-autonomous gradient flow~\eqref{alpha_gf} may be useful for analyzing other methods such as proximal methods. More generally, in the next section, we analyze  a family of  second-order gradient flows without adjusting the time scale (taking $\alpha_t =1$). 

\subsubsection{A non-autonomous second-order gradient flow}

Nesterov's accelerated gradient flow reaches an $\mathcal{O}(\frac{1}{t^2})$ convergence in function values (see Theorem~\ref{conv_theorem_order2}). We study convergence properties of a non-autonomous second-order gradient flow and compare them with those of the accelerated gradient flow, for the family of quadratic Lyapunov functions~\eqref{lyapunov_order2}. Let $\beta_t \geqslant 0$ be a continuous function, and a second-order non-autonomous gradient flow: 
\begin{equation}
\label{alphabeta_gf}
    \ddot{X}_t + \beta_t \dot{X}_t + \nabla f(X_t) = 0.
\end{equation}

\begin{remark}
Wibisono~et~al.~\cite[Theorem 2.2]{wibisono2016variational} proved that this ODE is related to the family of ODEs defined by $\dot{Y}_t + \Tilde{\beta}_tY_t + \Tilde{\alpha}_t\nabla f(Y_t)=0$. Let  $\alpha_t > 0$ be a continuously differentiable function with respect to time, $\tau_t = \int_0^t {\alpha_s}ds$ a time rescaling and $X_t$ a solution to~\eqref{alphabeta_gf}. The trajectory $Y_t = X_{\tau_t}$ is solution to $\ddot{Y}_t + (\beta_{\tau_t}\sqrt{\alpha_t} - \frac{\dot{\alpha}_t}{2\alpha_t})\dot{Y}_t + \alpha_t\nabla f(Y_t) = 0$. In contrast with $\alpha_t$ in~\eqref{alpha_gf}, note that changing $\beta_t$ in~\eqref{alphabeta_gf} does not correspond to a time rescaling.
\end{remark}

Theorem~\ref{acc_ode_conv} provides a systematic condition for the function $\mathcal{V}$ taken among the class of quadratic Lyapunov functions~\eqref{lyapunov_order2} to be a Lyapunov function for the second-order gradient flow~\eqref{alphabeta_gf}.

\begin{theorem}
\label{acc_ode_conv}
Let $a_t \geqslant 0$ and $P_t \succcurlyeq 0$ be continuously differentiable with respect to time, and $\mu \geqslant 0$. The following assertions are equivalent:
\begin{itemize}
    \item The inequality $\displaystyle\frac{d}{dt}\mathcal{V}_{a_t, P_t}(X_t,t) \leqslant 0$ is satisfied for all dimension $d \in \mathbf{N}$, all function $f \in \mathcal{F}_{\mu, \infty}$ and all trajectory $X_t$ generated by the non-autonomous second-order gradient flow~\eqref{alphabeta_gf}, where $\mathcal{V}_{a_t, P_t}$ is a quadratic Lyapunov function of the form~\eqref{lyapunov_order2}.
    \item There exist  $\lambda^{(1)}_t, \lambda^{(2)}_t \geqslant 0$ such that:
    \begin{equation}
    \label{eq:acc_LMI}
    \begin{aligned}
&\begin{pmatrix} -\frac{\mu}{2}(\lambda^{(1)}_t + \lambda^{(2)}_t) + \dot{p}_{t}^{(11)} & p^{(11)}_t - \beta_tp^{(12)}_t + \dot{p}^{(12)}_t & - p^{(12)}_t + \frac{\lambda_t^{(1)}}{2} \\ p^{(11)}_t - \beta_tp^{(12)}_t + \dot{p}^{(12)}_t & 2(p^{(12)}_t - \beta_tp^{(22)}_t) + \dot{p}^{(22)}_t & -p^{(22)}_t + \frac{a_t}{2} \\ - p^{(12)}_t + \frac{\lambda_t^{(1)}}{2} & -p^{(22)}_t + \frac{a_t}{2} & 0\end{pmatrix} \preccurlyeq 0, 
\\ &\dot{a}_t = \lambda^{(1)}_t - \lambda^{(2)}_t.
\end{aligned}
\end{equation}
\end{itemize}
\end{theorem}
The LMI~\eqref{eq:acc_LMI} from Theorem~\ref{acc_ode_conv} is parametrized by $\beta_t$. When $\beta_t = \frac{3}{t}$ and $\mu = 0$, we recover Theorem~\ref{conv_theorem_order2} for Nesterov's accelerated gradient flow.

\begin{corollary}
\label{coro_accparam}
Let $\mu \geqslant 0$. The function
\begin{equation*}
    \mathcal{V}(X_t, t) = a_t(f(X_t) - f_\star) + \frac{1}{2a_t}\|a_t \dot{X}_t + \dot{a}_t(X_t - x_\star)\|^2,
\end{equation*}
with $a_t$ defined by:
\begin{itemize}
    \item If $\mu > 0$, $a_t = e^{\tau t}$, with $\tau = \min(\sqrt{\mu}, \frac{2}{3}\beta_t)$,
    \item If $\mu=0$, $a_t = \min\left((\sqrt{a_0} + (\sqrt{p_0^{(11)}}/2)t)^2, \lim_{\epsilon \rightarrow 0, \\ \epsilon > 0}a_\epsilon e^{\int_\epsilon^t  \frac{2}{3}\beta_sds}\right)$,
\end{itemize}
verifies $\displaystyle\frac{d}{dt}\mathcal{V}(X_t,t) \leqslant 0$ for all dimension $d\in \mathbf{N}$, all function $f \in \mathcal{F}_{\mu, \infty}$ and all trajectory $X_t$ generated the second-order gradient flow~\eqref{alphabeta_gf}.
\end{corollary}
\begin{proof}
The proof follows from Theorem~\ref{acc_ode_conv} and is detailed in Appendix~\ref{sec:A00}.
\end{proof}

Given a convex function $f \in \mathcal{F}_{0, \infty}$ and a quadratic Lyapunov function, Corollary~\ref{coro_accparam} allows to conclude the following about the convergence of the second-order gradient flow~\eqref{alphabeta_gf}: given the class of quadratic Lyapunov functions~\eqref{lyapunov_order2}, it cannot converge faster in function values than Nesterov's accelerated gradient flow, i.e., not faster than $\mathcal{O}(\frac{1}{t^2})$.
\smallbreak
To analyze Nesterov's accelerated gradient methods using ODEs, Su et~al.~\cite{Su2015} introduced a parametrized second-order gradient flow, that fits the model~\eqref{alphabeta_gf}: 
\begin{equation}
\label{su_agf}
    \ddot{X}_t + \frac{r}{t}\dot{X}_t + \nabla f(X_t) = 0,
\end{equation}
where $r \geqslant 0$. When $r\geqslant 3$, the guarantee $f(X_t) - f_\star \leqslant \frac{(r-1)^2\|x_0 - x_\star\|^2}{2t^2}$ holds for any function $f \in \mathcal{F}_{0, \infty}$ and any trajectory $X_t$ generated by the accelerated gradient flow~\eqref{su_agf}~\cite[Theorem 5]{Su2015}. When $r<3$, Attouch et~al.~\cite[Theorem 2.1]{Attouch2019} proved a convergence bound in function values $f(X_t) - f_\star = \mathcal{O}(\frac{1}{t^{2r/3}})$. Using Corollary~\ref{coro_accparam}, we retrieve a similar bound in function values $f(X_t) - f_\star \leqslant \frac{2\|x_0 - x_\star\|^2r^2}{9t^{\frac{2r}{3}}}$.
\begin{remark}
Polynomial convergence can be achieved up to a time rescaling, as it was shown by Wisobono et~al.~\cite{wibisono2016variational}. Given $\tau_t = t^{p/2}$ ($\alpha_t = \frac{p}{2}t^{p/2 - 1}$), Nesterov's accelerated gradient flow ($r=3$) transforms into $\ddot{X}_t + \frac{p+1}{t}\dot{X}_t + \frac{p^2}{4}t^{p-2}\nabla f(X_t) = 0$ for $p \geqslant 2$. Corollary~\ref{coro_accparam} ensures convergence in function values $f(X_t) - f_\star \leqslant \frac{\|x_0 - x_\star\|^2}{2t^p}$, for all dimension $d \in \mathbf{N}$, all function $f \in \mathcal{F}_{0, \infty}$ and all trajectory $X_t$ generated by 
the second-order gradient flow~\eqref{alphabeta_gf}.
\end{remark}
\bigbreak
We have extended the performance estimation approach to continuous-time models using Lyapunov functions. Given a (possibly accelerated) gradient flow and a class of functions, we presented a semidefinite formulation equivalent with the existence of a certain type of quadratic Lyapunov function. The next section is devoted to the analysis of continuous-time models approximating SGD.

\section{SDEs for modeling SGD}
\label{sec:sde_for_sge}

 Convergence results for stochastic convex optimization often require additional assumptions on problem classes, refined choices of step sizes and averaged iterates. Their analyses raise more complex proofs in contrast with deterministic methods. Prior works have been concerned with connections between stochastic methods and stochastic differential equations (SDEs)~\cite{li2017stochastic,orvieto2020continuoustime,xu2018, shi2019acceleration, shi2018understanding}. This section is devoted to convergence analyses of SDEs approximating stochastic methods using a systematic Lyapunov approach. Verifying a small-sized LMI will be sufficient (but not necessary) for verifying a quadratic Lyapunov function.
\smallbreak
Stochastic gradient descent (SGD) is given by:
\begin{equation}
    x_{k+1} = x_k - \gamma\nabla \tilde{f}(x_k, \xi_{i_k}),
\end{equation} 
where $\gamma>0$ is the step size, $\xi_{i_k}$ are uniformly drawn in $(\xi_1, ..., \xi_n)$, and where $\nabla \tilde{f}(x_k, \xi_{i_k})$ is an unbiased estimate of full gradient $\nabla f(x_k)$. Li~et~al.~\cite{li2017stochastic} introduced stochastic modified equations (SMEs) to model SGD, rewriting it as
\begin{equation*}
   x_{k+1} = x_k - \gamma \nabla\tilde{f}(x_k, \xi_{i_k})+ \sqrt{\gamma}V_k(x_k),
\end{equation*}
where $V_k(x) = \sqrt{\gamma}\left(\nabla f(x_k) - \nabla \tilde{f}(x_k, \xi_{i_k})\right)$ has zero mean and a covariance matrix equal to $\gamma \Sigma(x_k) = \gamma (\sum_{i=1}^n (\nabla f(x_k) - \nabla \tilde{f}(x_k, \xi_{i_k}))(\nabla f(x_k) - \nabla \tilde{f}(x_k, \xi_{i_k}))^\top )$. 
\smallbreak
The corresponding SDE is given by:
\begin{equation}
\label{eq:sme}
    dX_t = -\nabla f(X_t)dt + (\gamma \Sigma(X_t))^{1/2}dB_t,
\end{equation}
where $B_t$ is a standard Brownian motion. The SDE~\eqref{eq:sme} is an (order-$1$ weak) approximation of  SGD~(\cite[Theorem 1]{li2017stochastic},~\cite{Li2019}), which allows to take into account the role of constant step size in the dynamics of SGD (while keeping them small). Under mild assumptions on $f$, Li~et~al.~\cite{Li2019} proved the weak approximation of SGD by this SDE on a finite interval $[0, T]$: there exists $C>0$ such that $\|\mathbf{E}[x_k] - \mathbf{E}[X(k\gamma)] \| \leqslant C\gamma$ for $k \in [0, \frac{T}{\gamma}]$. However, the approximation point of view from this approach is relatively limited since $C$ depends exponentially on~$T$. 

\begin{remark}
The SDE~\eqref{eq:sme} is an approximation of SGD for small step sizes~$\gamma \geqslant 0$. When the step size goes to zero, the noise term actually disappears, and the limiting ODE of SGD is exactly the gradient flow~\eqref{eq:gf_init}. In contrasts, the stochastic Langevin dynamics~\cite{langevin1908} $x_{k+1} = x_k - \gamma \nabla f(x_k) - \sqrt{\gamma}\xi_k$ has the limiting ODE $dX_t = -\nabla f(X_t)dt + \sqrt{2}dB_t$, where the step size is not taken into account.
\end{remark}

Compared to gradient descent, SGD does not converge to a stationary point under constant step sizes~\cite{bach2011, taylor2021stochastic}. Convergence to a stationary point requires diminishing step sizes such as $\gamma_k = \frac{1}{\sqrt{k}}$. Li~et~al.~\cite[Section 4.1]{li2017stochastic} (and later Orvieto~and~Lucchi~\cite[Section 2.1]{orvieto2020continuoustime}) proposed to include this varying learning rate in the dynamics:
\begin{equation*}
    x_{k+1} = x_k - \gamma h_k \nabla f(x_k),
\end{equation*}
where $\gamma$ is the maximum allowed learning rate and $h_k \in [0, 1]$ is the time-varying part. For $h_t \geqslant 0$ a continuous function playing the role of step sizes $h_k$, the SDE is:
\begin{equation}
\label{sme_general}
    dX_t = - h_t \nabla f(X_t)dt + h_t(\gamma \Sigma(X_t))^{1/2}dB_t.
\end{equation}
We treat the covariance matrix $\Sigma(X_t)$ as symmetric, already implied by the notation $\Sigma(X_t)^{1/2}$, but unstructured with bounded variance $\Sigma(X_t) \preccurlyeq \Sigma$ along any trajectory $X_t$ generated by the approximating SDE~\eqref{sme_general}. Compared to ODEs, functions $f \in \mathcal{F}_{0, L}$ to be optimized using SDEs are in addition assumed to be possibly $L$-smooth with $L\in (0, \infty]$ and to be twice continuously differentiable.
\smallbreak
In this section, we analyze approximating SDEs with averaging techniques, and include time-varying step sizes later. Verifying Lyapunov functions thanks to small-sized LMIs, we retrieve convergence results from discrete-time optimization methods, using appropriate choices of step sizes. 

\subsection{Lyapunov functions do not always extend to the stochastic setting}

The analysis of the gradient flow in the deterministic case provides Lyapunov functions that are decreasing along any trajectory generated by \eqref{eq:gf_init} (see Section~\ref{sec:deterministic}). The direct transfer of these Lyapunov functions to the stochastic setting is not always suited to the variance term, as detailed below. Under constant step sizes $\gamma > 0$ (and $h_t = 1$), an approximating SDE of SGD is:
\begin{equation*}
    dX_t = -\nabla f(X_t)dt + (\gamma \Sigma(X_t))^{1/2}dB_t.
\end{equation*}
\smallbreak
In SDE theory, a differential with respect to time of a function of a solution to a stochastic process is given by Ito's Lemma~\cite[Theorem 4.2]{sarkka_solin_2019}.
\begin{lemma}[Ito's Lemma]
Let $g$ be a twice continuously differentiable function, and $X_t$ be a stochastic process solution to the SDE~\eqref{eq:sme}, then
\begin{equation*}
    dg(X_t,t) = \frac{\partial}{\partial t}g(X_t, t)dt + \frac{\partial}{\partial x}g(X_t, t) dX_t + \frac{1}{2}\gamma {\rm Tr}(\frac{\partial^2}{\partial x^2}g(X_t, t)\Sigma(X_t))dt.
\end{equation*}
\end{lemma}

\subsubsection{Minimizing strongly convex functions}

When the SDE originates from (possibly non-smooth) strongly convex functions $f\in \mathcal{F}_{\mu, \infty}$, the Lyapunov function from the deterministic setting extends well to SDEs. In the deterministic setting, we have shown in Theorem~\ref{conv_gf_mu} that the function $\mathcal{V}(x,t) = e^{2 \mu t}(f(x)-f_\star)$ is a Lyapunov function along the gradient flow.
Given $X_t$ a solution to the SDE~\eqref{eq:sme}, $\frac{d}{dt}\mathbf{E}\mathcal{V}(X_t,t) \leqslant \frac{1}{2}e^{2 \mu t} \gamma \mathbf{E}{\rm Tr}(\nabla^{2}_{xx} f(X_t)\Sigma(X_t))$ follows from Ito's formula. After integration between $0$ and $t$, we have:
\begin{equation*}
    \mathbf{E}(f(X_t)-f_\star) \leqslant e^{-2\mu t}\left( f(x_0) - f_\star  + \frac{1}{2} \gamma \int_0^t  e^{2\mu s} \mathbf{E}{\rm Tr}(\nabla^{2}_{xx} f(X_s)\Sigma(X_s))ds \right).
\end{equation*}

We cannot conclude about the convergence of the trajectory $X_t$ without additional requirements on the problem class. For example, let the smoothness parameter of $f$ be finite $L < \infty$ and recalling the bounded covariance assumption $\Sigma(X_t) \preccurlyeq \Sigma$, the variance term is bounded by $\frac{1}{2}L\gamma {\rm Tr}(\Sigma)$. Therefore, under constant step sizes, the SDE approximating SGD converges to a diffusion. In addition, it cannot get to a stationary point $x_\star$, because of the extra term, which depends linearly in the step size~$\gamma$.  As in the deterministic case, the forgetting of initial conditions remains of order~$\mathcal{O}(e^{-2\mu t})$. 

\subsubsection{Minimizing convex functions}
When $f \in \mathcal{F}_{0, \infty}$, Lyapunov functions induce convergence bounds with a possibly diverging variance term. When considering the deterministic gradient flow~\eqref{eq:gf_init} originating from functions $f \in \mathcal{F}_{0, \infty}$, the Lyapunov function $\mathcal{V}(x,t) = t(f(x)-f_\star) + \frac{1}{2}\|x - x_\star\|^2$ followed from Theorem~\ref{conv_gf_conv}. Let $f \in \mathcal{F}_{0, \infty}$ be a twice continuously differentiable function, $X_t$ be any trajectory solution to the SDE~\eqref{eq:sme}. Thanks to Ito's formula, we compute the derivative of $\mathcal{V}(\cdot)$ with respect to time as follows: $\frac{d}{dt}\mathbf{E}\mathcal{V}(X_t,t) \leqslant -t \mathbf{E}\|\nabla f(X_t)\|^2 + \mathbf{E}\frac{1}{2}{\rm Tr}((t\nabla_x^2f(X_t) + I)\Sigma(X_t))$. Using the bounded covariance assumption $\Sigma(X_t) \preccurlyeq \Sigma$ and assuming in addition $f$ to be $L$-smooth with $L < +\infty$, a convergence bound in function values is given by:
\begin{equation*}
   \mathbf{E} (f(X_t)-f_\star) \leqslant \frac{\|x_0 - x_\star\|^2}{t} +  \frac{1}{2}(L\frac{t}{2} + 1){\rm Tr}(\Sigma).
\end{equation*}
Such an inequality does not allows to conclude about convergence of the SDE~\eqref{sme_general} without further assumptions.
\smallbreak
In discrete-time, Taylor~and~Bach proved a comparable convergence bound for SGD~\cite[Theorem 5]{taylor2021stochastic}, applying the Lyapunov performance estimation approach under similar assumptions (bounded variance, smoothness of $f$). Optimization methods and alternative techniques have been developed to ensure the global convergence of SGD to the optimum, among them averaging~\cite{ruppert1988} and diminishing step sizes.

\subsection{Diminishing the step size is a key to success}

We study convergence of SDEs with time-varying step sizes~\eqref{sme_general}. In contrast to the deterministic setting, in which time-varying step sizes correspond to a time rescaling whose benefit disappears after discretization, such a time rescaling plays a direct role in the variance term (explicit formula by Orvieto~and~Lucchi in~\cite[Theorem 5]{orvieto2020continuoustime}). 
\smallbreak
Let $f \in \mathcal{F}_{0, \infty}$ be a twice continuously differentiable function, $X_t$ be generated by~\eqref{sme_general}, and $\mathcal{V}$ be a function. Our goal is to control the maximization problem
\begin{equation*}
    \begin{aligned}
     \max_{X_t \in \mathbf{R}^d,  d\in \mathbf{N}, \ f\in \mathcal{F}_{0, \infty}} \ & \ \frac{d}{dt}\mathbf{E}\mathcal{V}(X_t,t),\\
      \text{subject to }
    \ dX_t &= -h_t\nabla f(X_t)dt + h_t(\gamma \Sigma(X_t))^{1/2} dB_t,
\end{aligned}
\end{equation*}
where $\frac{d}{dt}\mathbf{E}\mathcal{V}(X_t,t) = \mathbf{E}[\frac{\partial}{\partial t}\mathcal{V}(X_t, t)  + \frac{\partial}{\partial x}\mathcal{V}(X_t, t)\frac{dX_t}{dt}]+ \frac{\gamma}{2} \mathbf{E}{\rm Tr}(\frac{\partial^2}{\partial x^2}\mathcal{V}(X_t, t)\Sigma(X_t))$ is computed using Ito's formula. The first two terms correspond exactly to taking the derivative in trajectories generated by ODEs~\eqref{alpha_gf} (or the SDEs~\eqref{sme_general} with $\gamma=0$), and the last term corresponds to a variance term. Because of the trace in a second-order derivative of $f$ and in the covariance matrix $\Sigma(X_t)$, we do not take this term into account in an LMI formulation. Instead, we propose to first derive a family of Lyapunov functions for associated ODEs using LMIs (deterministic setting), and then to optimize their parameters so that the variance term converges conveniently. In other words, the following functions $\mathcal{V}$ are Lyapunov functions for induced ODEs ($\gamma = 0$), but are not Lyapunov functions for SDEs under consideration. This approach allows a systematic computation of such quadratic functions and leads to convergence guarantees.

\begin{corollary}
\label{pr_step_size_tr}
Let $f \in \mathcal{F}_{0, \infty}$ be a twice continuously differentiable function, and $X_t \in \mathbf{R}^d$ be generated by the SDE~\eqref{sme_general}. The quadratic function 
\begin{equation*}
    \mathcal{V}(X_t, t) = a^{(1)}_t(f(X_t) - f_\star) + \frac{1}{2}\|X_t - x_\star\|^2,
\end{equation*}
with $\dot{a}_t^{(1)} = 2 h_t$ verifies $\displaystyle\frac{d}{dt}\mathbf{E}(\mathcal{V}(X_t,t) \leqslant h_t^2\mathbf{E}{\rm Tr}((\nabla_{xx}^2f(X_t) a^{(1)}_t + \frac{1}{2}I_d)\Sigma(X_t))$. 

Furthermore, it holds that: 

$\displaystyle \mathbf{E}[f(X_t) - f_\star] \leqslant \frac{\|x_0 - x_\star \|^2}{a^{(1)}_t} + \frac{\gamma}{2a^{(1)}_t}\int_0^t  h_s^2\mathbf{E}{\rm Tr}((\nabla_{xx}^2f(X_s) a^{(1)}_s + \frac{1}{2}I_d)\Sigma(X_s))ds$.
\end{corollary}
\begin{proof}
The function $\mathcal{V}$ is obtained using Corollary~\ref{alpha_coro_1}, and is a Lyapunov function for a non-autonomous first-order gradient flow. The bound for $\mathbf{E}[f(X_t) - f_\star]$ is derived using Ito's formula on $\mathcal{V}$ along trajectories $X_t$ generated by the SDE~\eqref{sme_general}.
\end{proof}
The convergence bound from Corollary~\ref{pr_step_size_tr} is divided in two terms: a term that forgets the initial conditions and a variance term due to noise. Convergence is mostly controlled by the step size ($\dot{a}^{(1)}_t = 2 h_t$). Bach~and~Moulines~\cite[Theorem 5]{bach2011} provided a comparable but much more complex analysis for stochastic gradient descent, for a specific family of step sizes. To compare to our results, let us consider step sizes defined by $h_t = \frac{1}{(t+1)^{\alpha}}$, where $\alpha \geqslant 0$. A possible convergence guarantee arises from Corollary~\ref{pr_step_size_tr} with $a^{(1)}_t = (t+1)^{1-\alpha}$. The forgetting of the initial condition is thus bounded by $\frac{\|x_0 - x_\star \|^2}{(t+1)^{1-\alpha}}$. Provided $\Sigma(X_t) \preccurlyeq \Sigma$, the variance term is bounded by:
\begin{align*}
\left\{
   \begin{array}{ll}
      \frac{\gamma {\rm Tr}(\Sigma)}{(t+1)^{1 - \alpha}} \left( L \frac{(t+1)^{2 - 3\alpha} - 1}{2 - 3 \alpha}  + \frac{1}{2} \frac{(t+1)^{1 - 2\alpha} - 1}{1-2\alpha }\right) & \mbox{if } 0\leqslant \alpha < 1, \\
     \frac{\gamma {\rm Tr}(\Sigma)}{\text{log}(t)}\left(L\int_0^t \frac{\text{log}(s+1)}{(s+1)^2}ds + \frac{1}{2}(1 - \frac{1}{t+1})\right)& \mbox{if } \alpha = 1.
        \end{array}\right.
\end{align*}

The variance term converges to zero if and only if $\alpha \geqslant \frac{1}{2}$. In other words, convergence is not guaranteed for constant step sizes ($\alpha=0$). For $\alpha \in (1/2, 2/3)$, the convergence in function value is bounded by $\mathcal{O}(\frac{1}{t^{2\alpha - 1}})$. For $\alpha \in (2/3, 1)$, the convergence in function value is bounded by $\mathcal{O}(\frac{1}{t^{1-\alpha}})$. As for SGD~\cite[Theorem 5]{bach2011}, the convergence regime changes at $\alpha = \frac{2}{3}$ with a global convergence rate in $\frac{1}{t^{1/3}}$, for which the variance term and the term that forgets the initial conditions converge at the same rate (up to $\text{log}(t)$). It is therefore possible to reach convergence with diminishing step sizes. Other techniques, such as averaging, have been developed to improve the trade-off between faster convergence and larger step sizes.

\subsection{Averaging for larger step sizes}
Polyak-Ruppert averaging~\cite{ruppert1988,polyak1992} is a standard way to improve convergence of SGD. In the discrete-time setting, convergence guarantees are considered at an averaged sequence, where $i_k$ is drawn uniformly in $[1, ..., n]$ and $n$ the sample size:
\begin{equation}
\label{eq:pr_avg}
\begin{aligned}
    x_{k+1} &= x_k - \gamma h_k\nabla f_{i_k}(x_k), \\
    \bar{x}_k &= \frac{1}{k}\sum_{i=1}^kx_i.
\end{aligned}
\end{equation} 
Other averaging techniques were developed later, such as primal averaging~\cite{tao2018} and averaging with respect to some nonnegative function~\cite{leroux2019}. 

\subsubsection{Polyak-Ruppert averaging}
In this section, we analyze convergence properties of an SDE~\eqref{sme_general} approximating SGD with time-varying step sizes under Polyak-Ruppert averaging. Taylor~and~Bach~\cite[Theorem 6]{taylor2021stochastic} provided a systematic design of Lyapunov functions for~\eqref{eq:pr_avg}, and a condition on the step size for convergence to the optimum. An approximating SDE for Polyak-Ruppert averaging is given by:
\begin{equation}
\label{ODE_mean1}
    \begin{aligned}
        dX_t &= -h_t\nabla f(X_t)dt + h_t(\gamma \Sigma(X_t))^{1/2} dB_t, \\
        d\bar{X}_t &= \frac{X_t - \bar{X}_t}{t}dt,
        \end{aligned}
    \end{equation}
with step size $\gamma>0$ that is taken close to zero, and a variable term $h_t\in [0, 1]$. We introduce the family of quadratic functions taking the averaged sequence into account:
\begin{equation}
\label{lyap_avg}
    \mathcal{V}_{a^{(1)}_t, a^{(2)}_t, P_t}(X_t, \bar{X}_t, t) = \begin{pmatrix}a^{(1)}_t \\ a^{(2)}_t \end{pmatrix}^\top \!\!\!\begin{pmatrix} f(X_t) - f_\star \\ f(\bar{X_t}) - f_\star \end{pmatrix} + \begin{pmatrix} X_t - X_\star\\ \bar{X_t} - X_\star\end{pmatrix}^\top \!\!\! \left(P_t\otimes I_d\right) \! \begin{pmatrix} X_t - X_\star\\ \bar{X_t} - X_\star\end{pmatrix},
\end{equation}
where $a^{(1)}_t, a^{(2)}_t \geqslant 0$ and $P_t = \begin{pmatrix}p^{(11)}_t & p^{(12)}_t \\ p^{(12)}_t & p^{(22)}_t \end{pmatrix}\succcurlyeq 0$ are continuously differentiable with respect to time. Given a quadratic function $ \mathcal{V}_{a^{(1)}_t, a^{(2)}_t, P_t}$ and an SDE~\eqref{ODE_mean1} under Polyak-Ruppert averaging, we present in Theorem~\ref{pr_avg_conv} a  way to control the quantity:
\begin{equation*}
    \begin{aligned}
      \max_{f \in \mathcal{F}_{0, \infty}, d \in \mathbf{N}, X_t \in \mathbf{R}^d} \ & \ \frac{d}{dt}\mathbf{E}\mathcal{V}_{a^{(1)}_t, a^{(2)}_t, P_t}(X_t, \bar{X}_t, t), \\
       \text{subject to} \ & dX_t = -h_t\nabla f(X_t)dt + h_t(\gamma \Sigma(X_t))^{1/2} dB_t, \\
        & d\bar{X}_t = \frac{X_t - \bar{X}_t}{t}dt.
    \end{aligned}
\end{equation*}

\begin{theorem}
\label{pr_avg_conv}
Let $h_t \geqslant 0$, and $a_t \geqslant 0$, $P_t \succcurlyeq 0$ be continuously differentiable with respect to time. Let  $f \in \mathcal{F}_{0, \infty}$ be a twice continuously differentiable function and $(X_t, \bar{X}_t) \in \mathbf{R}^d \times \mathbf{R}^d$ be a trajectory generated by the SDE under Polyak-Ruppert averaging~\eqref{ODE_mean1}. Let $\mathcal{V}$ be a quadratic function~\eqref{lyap_avg}, such that there exist $\lambda^{(1)}_t, ..., \lambda^{(6)}_t \geqslant 0$ verifying:
\begin{equation*}
    \begin{aligned}
&\begin{pmatrix} \dot{p}^{(11)}_t + \frac{2p^{(12)}_t}{t} & \dot{p}^{(12)}_t + \frac{p^{(22)}_t - p^{(12)}_t}{t}& \frac{\lambda^{(6)}_t + \lambda^{(4)}_t - 2h_tp^{(11)}_t}{2} & \frac{a^{(2)}_t}{2t} - \frac{\lambda^{(5)}_t}{2}\\ 

\dot{p}^{(12)}_t + \frac{p^{(22)}_t - p^{(12)}_t}{t} & \dot{p}^{(22)}_t - \frac{2p^{(22)}_t}{t} & -\frac{\lambda^{(6)}_t - 2h_tp^{(12)}_t}{2} & -\frac{a^{(2)}_t}{2t} + \frac{\lambda^{(5)}_t + \lambda^{(3)}_t}{2}\\ 

 \frac{\lambda^{(6)}_t + \lambda^{(4)}_t - 2h_tp^{(11)}_t}{2} & -\frac{\lambda^{(6)}_t - 2h_tp^{(12)}_t}{2} & -a^{(1)}_t & 0 \\ 

\frac{a^{(2)}_t}{2t} - \frac{\lambda^{(5)}_t}{2} & -\frac{a^{(2)}_t}{2t} + \frac{\lambda^{(5)}_t + \lambda^{(3)}_t}{2} & 0 &  0\end{pmatrix} \preccurlyeq 0, \\ 
&\dot{a}^{(1)}_t + \lambda^{(1)}_t + \lambda^{(5)}_t = \lambda^{(4)}_t + \lambda^{(6)}_t, \\ 
&\dot{a}^{(2)}_t + \lambda^{(2)}_t + \lambda^{(6)}_t = \lambda^{(3)}_t + \lambda^{(5)}_t + \dot{a}^{(1)}_t. 
\end{aligned}
\end{equation*}
Then, the following inequality is satisfied:
\smallbreak 
$\displaystyle\frac{d}{dt}\mathbf{E}\mathcal{V}(X_t, \bar{X}_t, t)\leqslant \frac{1}{2}\mathbf{E}{\rm Tr}\left((a^{(1)}_t\nabla_{xx} f(X_t) + 2p^{(11)}_t I_d)\Sigma(X_t)\right)h_t^2\gamma$.
\end{theorem}
\begin{proof}
The proof follows the method from Section~\ref{sec:methodo} (see Appendix~\ref{sec:A11}).
\end{proof}

The variance term $\frac{1}{2}\mathbf{E}{\rm Tr}((a^{(1)}_t\nabla_{xx} f(X_t) + 2p^{(11)}_t I_d)\Sigma(X_t))h_t^2\gamma$ from Theorem~\ref{pr_avg_conv} increases with $a^{(1)}_t$, and its convergence requires additional assumptions on $f$, for example smoothness. For this reason, we propose to analyze convergence guarantees based on functions $\mathcal{V}_{0, a^{(2)}_t, P_t}$, on the averaged sequence only.

\begin{corollary}[Averaging and diminishing step sizes]
\label{pr_step_size_tr_avg}
Let $h_t \geqslant 0$ and $a_t^{(2)} \geqslant 0$ be continuously differentiable. Let $f \in \mathcal{F}_{0, \infty}$ be a twice continuously differentiable function and $(X_t, \bar{X}_t)$ be a trajectory generated by the SDE under Polyak-Ruppert averaging~\eqref{ODE_mean1}. Assuming that $\dot{a}^{(2)}_t \leqslant \frac{a^{(2)}_t}{t}$ and $t \rightarrow \frac{a^{(2)}_t}{th_t} $ is a non-increasing function, the function
\begin{equation*}
    \mathcal{V}(X_t, \bar{X}_t, t) = a^{(2)}_t(f(\bar{X_t}) - f_\star) + \frac{a^{(2)}_t}{2 t h_t}\|X_t - X_\star\|^2
\end{equation*} verifies $\frac{d}{dt}\mathbf{E}\mathcal{V}(X_t, \bar{X}_t,t) \leqslant\frac{a^{(2)}_t}{t}h_t\mathbf{E}{\rm Tr}( \Sigma(X_t))$. Furthermore, it holds that $\displaystyle \mathbf{E}[f(\bar{X}_t) - f_\star] \leqslant \frac{\|x_0 - x_\star \|^2}{2 a^{(2)}_t} + \frac{\gamma}{2a^{(2)}_t}\int_0^t  \frac{a^{(2)}_s}{s}h_s\mathbf{E}{\rm Tr}( \Sigma(X_s))ds$.
\end{corollary}
\begin{proof}
The proof follows from the LMI in Theorem~\ref{pr_avg_conv} with the choices $\lambda_t^{(3)} =\lambda_t^{(6)} = \lambda_t^{(2)} =0$, 
$\lambda_t^{(5)} = \frac{a_t^{(2)}}{t}$, 
$\lambda_t^{(4)} = \lambda_t^{(1)} =  h_tp_t^{(11)}$, $\dot{a}_t^{(2)} = \frac{a_t^{(2)}}{t}$, $\dot{p}_{t}^{(11)} = \dot{(\frac{a_t^{(2)}}{th_t})}$, $\dot{p}_{t}^{(12)}= \dot{p}_{t}^{(22)}=0$.
\end{proof}

When $a^{(2)}_t = t$ (its maximal possible value), the step size verifies $\dot{h}_t \leqslant 0$. The variance term does not diverge if and only if $h_t$ is constant. Recalling the unbounded covariance $\Sigma_t \preccurlyeq \Sigma$, a convergence bound is given by $\mathbf{E}[f(\bar{X}_t) - f_\star] \leqslant \frac{\|x_0 - x_\star \|^2}{2t} + \frac{1}{2} {\rm Tr}(\Sigma)\gamma h$. The decreasing condition on $t \rightarrow \frac{a^{(2)}_t}{th_t}$ suggests a trade-off between convergence and diminishing step size, as obtained without averaging. 
\smallbreak
Under the assumptions of Corollary~\ref{pr_step_size_tr_avg}, recalling that $\Sigma(X_t) \preccurlyeq \Sigma$ is a bounded covariance matrix, let $h_t = \frac{1}{(t+1)^\alpha}$ be the step size and $a_t^{(2)} = t^\beta$, where $\alpha \geqslant 0$ and $ 0 \leqslant \beta \leqslant 1$. The decreasing condition imposes $\alpha + \beta \leqslant 1$. In comparison with step sizes requirements drawn from Corollary~\ref{pr_step_size_tr}, where convergence required $\alpha \geqslant \frac{1}{2}$, averaging allows larger step sizes. 
\smallbreak
A different behavior is expected from $\alpha$ and $\beta$: on the one hand, an ideal step size should be large ($\alpha$ small), and on the other hand, we aim at converging as fast as possible ($\beta$ large). The term that forgets the initial conditions behaves as $\mathcal{O}(\frac{1}{t^\beta})$, and the variance term as $\mathcal{O}(\frac{1}{t^\alpha})$ if $\beta \neq \alpha$. When $\alpha = \beta$, the variance term behaves as $\mathcal{O}(\frac{\text{log}(t)}{2t^\beta})$. Hence, a natural choice is $\alpha = \beta = \frac{1}{2}$, retrieving results from~\cite[Theorem 4]{bach2011}~\cite[Table 2]{taylor2021stochastic} in discrete-time optimization.

\subsubsection{Weighted averaging}

Polyak-Ruppert averaging performs uniform averaging of any trajectory $X_t$ over the time step. We introduce weighted averaging to analyze SGD, that is defined with respect to a function $u_t \geqslant0$~\cite{leroux2019}:
\begin{equation*}
        \bar{x}^u_t = \frac{1}{\int_0^t  u_sds}\int_{0}^t u_s x_s ds.
\end{equation*}
Under weighted averaging, and introducing $C^u_t = \frac{u_t}{\int_0^t  u_sds}$, the SDE is given by:
\begin{equation}
\label{ODE_mean_u}
    \begin{aligned}
        & dX_t = -h_t\nabla f(X_t)dt + h_t(\gamma \Sigma(X_t))^{1/2} dB_t, \\
        & d\bar{X}^u_t = (X_t - \bar{X}^u_t)C^u_tdt.
        \end{aligned}
    \end{equation}
We study convergence of this generalized version of Polyak-Ruppert averaging, and compare it to traditional averaging techniques.

\begin{theorem}
\label{pr_avg_conv_u_1}
Let $\mu \geqslant 0$. Let $d \in \mathbf{N}$ be the dimension, $f \in \mathcal{F}_{\mu, \infty}$ be a twice continuously differentiable (possibly strongly convex) function and $(X_t, \bar{X}_t)$ be trajectories generated by an SDE under generalized averaging~\eqref{ODE_mean_u}. Assuming $\dot{\left(\frac{\int_0^tu_sds}{2h_t}\right) } 2\mu u_t$, the function
\begin{equation*}
    \mathcal{V}(X_t,\bar{X}^u_t, t) = \frac{u_t}{C^u_t}(f(\bar{X}_t^u) - f_\star) + \frac{u_t}{2h_t}\|X_t - x_\star\|^2,
\end{equation*}
 verifies $\frac{d}{dt}\mathbf{E}\mathcal{V}(X_t,\bar{X}^u_t, t)\leqslant \frac{1}{2}u_th_t\gamma\mathbf{E}{\rm Tr}(\Sigma(X_t))$.
\smallbreak
\noindent Then, it holds that $\displaystyle  \mathbf{E}[f(\bar{X}^u_t) - f_\star] \leqslant \frac{\|x_0 - x_\star \|^2u_0}{2h_0\int_0^t  u_sds} + \frac{\gamma}{4\int_0^t  u_sds}\int_0^t  u_sh_s\mathbf{E}{\rm Tr}(\Sigma(X_s))ds$.
\end{theorem}
\begin{proof}
The proof follows those from Theorem~\ref{pr_avg_conv} and Corollary~\ref{pr_step_size_tr_avg} with $\frac{1}{t} \rightarrow C^u_t$.
\end{proof}

Under the convexity assumption ($\mu=0$), $u_t$ verifies $\dot{\left(\frac{\int_0^tu_sds}{2h_t}\right)} \leqslant 0$. Assuming step sizes of the form $h_t = \frac{1}{(t+1)^\alpha}$ and an averaging function $u_t = \frac{1}{(t+1)^\beta}$, with $\alpha, \beta \geqslant 0$, it follows that $\beta \geqslant \alpha$. Recall that the covariance matrix $\Sigma(X_t) \preccurlyeq \Sigma$ is bounded. From Theorem~\ref{pr_avg_conv_u_1}, the terms that forgets the initial conditions behaves as $\mathcal{O}(\frac{1}{(t+1)^{1-\beta}})$, and the variance term as $\mathcal{O}(\frac{1}{(t+1)^\alpha})$. Both terms converge at same rate for $\alpha = \beta = \frac{1}{2}$ (up to $\log{(t+1)}$). In discrete time, similar convergence results have been derived for SGD under Polyak-Ruppert averaging~\cite[Theorem 6]{bach2011}.

\smallbreak
Under the strong convexity assumption ($\mu > 0$), polynomial convergence can be reached for the term that contains the initial conditions. However, the variance term cannot converge faster than the step size. Recalling that the covariance matrix $\Sigma(X_t) \preccurlyeq \Sigma$ is bounded, if $u_t = (t+1)^\beta$ and $h_t = (t+1)^{-\alpha}$ with $\alpha, \beta \geqslant 0$, the condition  $\dot{\left(\frac{\int_0^tu_sds}{2h_t}\right)} \leqslant 2\mu u_t$ implies that $\alpha = 0$. The variance term is then exactly equal to the step size $h_t=h_0$. Hence, weighted averaging allows a better convergence for the terms containing initial conditions, but does not play a role in the variance term. To conclude, weighted averaging does not improve convergence results obtained under Polyak-Ruppert averaging. The trade-off between the forgetting of the initial conditions and the noise term mostly relies on step sizes. 
\smallbreak

We have analyzed convergence of SGD together with averaging techniques using approximating SDEs~\eqref{sme_general}. Continuous-time analyses lead to similar convergence results, while benefiting from simpler formulations and fewer assumptions especially on step sizes. Using this approach, we analyzed the trade-off between nonuniform averaging and step sizes, paving the way to a better understanding of averaging techniques. In the next session, we explore new convergence analyses for stochastic accelerated methods.

\section{Accelerating the gradient flow}
\label{sec:extensions}

For both stochastic and deterministic models, we have retrieved known convergence results for continuous-time models approximating optimization methods. In this section, we provide convergence guarantees for a family of second-order gradient flows, including in particular AGF~\eqref{eq:nest_flow}.
\smallbreak
In the deterministic setting, convergence of gradient descent was improved using a momentum. In this section, let $f \in \mathcal{F}_{0, \infty}$ be a twice continuously differentiable function and $\gamma >0$ be constant step sizes. An approximating SDE (for order-1 weak approximations) for Nesterov's accelerated gradient is given by~\cite[Theorem~16,~Section~4.4]{Li2019}:
\begin{equation}
\label{eq:acc_gf}
    d^2X_t + \frac{3}{t} dX_t + \nabla f(X_t)dt + \sqrt{\gamma \Sigma(X_t)}dB_t = 0.
\end{equation}
 As for SGD, the function $\mathcal{V}(x, \dot{x}, t) = t^2(f(x) - f_\star) + 2 \|(x-x_\star) + \frac{t}{2}\dot{x}\|^2$ obtained from Theorem~\ref{conv_theorem_order2} does not allow to conclude about convergence to a stationary point of trajectories $X_t$ generated by the stochastic accelerated gradient flows~\eqref{eq:acc_gf}. Using the bounded covariance assumption $\Sigma(X_t) \preccurlyeq \Sigma$ and applying Ito's formula to $\mathcal{V}(\cdot)$ along $X_t$:
\begin{equation*}
    \mathbf{E}[f(X_t) - f_\star] \leqslant \frac{\|x_0 - x_\star\|^2}{t^2} + \gamma {\rm Tr}(\Sigma)3t.
\end{equation*}
In the following, we explore Polyak-Ruppert averaging together with diminishing step sizes, to analyze convergence of second-order SDEs.

\subsection{Averaging does not preserve convergence rates}
Averaging was a key to success for improving convergence of SGD (see Section~\ref{sec:sde_for_sge}). It is natural to wonder if averaging preserves the acceleration of Nesterov's gradient flow~\cite{Su2015}. Let us define the stochastic accelerated gradient flow under Polyak-Ruppert averaging:
\begin{equation}
\label{eq:acc_gf_avg}
\begin{aligned}
    & d^2X_t + \frac{3}{t} dX_t + \nabla f(X_t)dt + \sqrt{\gamma \Sigma}dB_t = 0, \\
    & d\bar{X}_t = \frac{X_t - \bar{X}_t}{t}dt.
\end{aligned}
\end{equation}
The family of quadratic functions into consideration is given by:
\begin{equation}
\label{lyapunov_acc_sto}
    \mathcal{V}_{a^{(1)}_t, a^{(2)}_t, P_t}(X_t, \bar{X}_t, t) = \begin{pmatrix} a^{(1)}_t \\ a^{(2)}_t \end{pmatrix}^\top \!\!\! \begin{pmatrix} f(X_t) - f_\star\\ f(\bar{X_t}) - f_\star\end{pmatrix} + \begin{pmatrix} \dot{X_t} \\  X_t - X_\star \\ \bar{X_t} - X_\star\end{pmatrix}^\top  \!\!\! \left(P_t\otimes I_d\right) \begin{pmatrix} \dot{X_t} \\  X_t - X_\star \\ \bar{X_t} - X_\star \end{pmatrix},
\end{equation}
where $P_t \succcurlyeq 0$ and $a^{(1)}_t, a^{(2)}_t \geqslant 0$ are continuously differentiable functions.

\begin{theorem}
\label{acc_sde_theorem}
Let $a^{(1)}_t, a^{(2)}_t \geqslant 0$ and $P_t \succcurlyeq 0$ be continuously differentiable with respect to time. Let $f \in \mathcal{F}_{0, \infty}$ be a twice continuously differentiable function and $(X_t, \bar{X}_t)$ be a trajectory generated by the stochastic accelerated gradient flow under Polyak-Ruppert averaging~\eqref{eq:acc_gf_avg} with constant step sizes ($h_t = 1$). Let in addition  $\mathcal{V}=\mathcal{V}_{a^{(1)}_t, a^{(2)}_t, P_t}(\cdot)$ be a quadratic function~\eqref{lyapunov_acc_sto}. Then it holds that:
\begin{itemize}
    \item When $a^{(1)}_t=0$, if the function $\mathcal{V}$ verifies $\frac{d}{dt}\mathbf{E}\mathcal{V}(X_t,\bar{X}_t, t)\leqslant {\rm Tr}(2p^{(11)}_t\gamma \Sigma(X_t))$, then $\mathcal{V}=0$.
    \item When $a^{(2)}_t=0$, the function
    \begin{equation*}
        \mathcal{V}(X_t, t) = a^{(1)}_t(f(X_t) - f_\star) + \frac{1}{2a^{(1)}_t}\|a^{(1)}_t \dot{X}_t + \dot{a}^{(1)}_t(X_t - x_\star)\|^2,
    \end{equation*}
    verifies $\frac{d}{dt}\mathbf{E}\mathcal{V}(X_t,t)\leqslant \mathbf{E}{\rm Tr}(2a^{(1)}_t\gamma \Sigma(X_t))$, with $a^{(1)}_t \leqslant t^2$.
    Furthermore, it holds that $\displaystyle \mathbf{E}[f(X_t) - f_\star] \leqslant \frac{2\|x_0 - x_\star \|^2}{a^{(1)}_t} + \frac{1}{2a^{(1)}_t}\gamma \int_0^t  a^{(1)}_s \mathbf{E}{\rm Tr}(\Sigma(X_s))ds.$
\end{itemize}
\end{theorem}
\begin{proof}
The proof follows from Theorem~\ref{acc_ode_conv} (see Appendix~\ref{sec:A21}). The first statement holds when considering a function $\beta_t \geqslant 0$ instead of $\frac{3}{t}$.
\end{proof}
Without further assumptions on the covariance $\Sigma_t$, the accelerated gradient flow under Polyak-Ruppert averaging~\eqref{eq:acc_gf_avg} with constant step sizes admits no quadratic function that allows to forget the initial conditions while reducing the variance term. Therefore, Polyak-Ruppert averaging plays a different role in the stochastic accelerated gradient flow compared to the stochastic gradient flow approximating SGD~\eqref{ODE_mean1}.

\subsection{Diminishing step sizes do not help preserve acceleration}

Averaging under constant step sizes was not conclusive for finding a convergence guarantee for Nesterov's accelerated gradient flow with a diffusion term. We consider a second-order non-autonomous stochastic gradient flow with time-varying step sizes $h_t \geqslant 0$,
\begin{equation}
\label{acc_step_ode}
    d^2X_t + \beta_t dX_t + h_t\nabla f(X_t)dt + h_t\sqrt{\gamma\Sigma(X_t)}dB_t = 0.
\end{equation}
To study the convergence of ~\eqref{acc_step_ode}, we generate a quadratic function within the class~\eqref{lyapunov_order2}.

\begin{theorem}
\label{theo_accstepsde}
Let $h_t \geqslant 0$, $\gamma > 0$ and $a_t \geqslant 0$ be continuously differentiable with respect to time. Let $f\in \mathcal{F}_{0, \infty}$ be a twice continuously differentiable function and $X_t$ be a trajectory generated by the stochastic second-order gradient flow~\eqref{acc_step_ode}. Assuming $\dot{a}_t \leqslant a_t \frac{2}{3}(\beta_t + \frac{1}{2} \frac{\dot{h}_t}{h_t})$ and $t\rightarrow \frac{(\dot{a}_t)^2}{2h_ta_t}$ a decreasing function, the function
\begin{equation*}
    \mathcal{V}(X_t) = a_t(f(X_t)-f_\star) + \frac{1}{2h_ta_t} \|a_t \dot{X}_t + \dot{a}_t(X_t - x_\star) \|^2,
\end{equation*}
verifies $\frac{d}{dt}\mathbf{E}\mathcal{V}(X_t,t) \leqslant \frac{1}{4}\mathbf{E}{\rm Tr}(a_th_t\Sigma(X_t))\gamma$.
\end{theorem}
\begin{proof}
This result is obtained by extending the LMI for ODEs from Theorem~\ref{acc_ode_conv} and Corollary~\ref{coro_accparam} to time-varying step sizes.
\end{proof}

To compare to previous results, we consider parametrized step sizes $h_t = \frac{1}{(t+1)^\alpha}$ and SDEs with $\beta_t=\frac{b}{t}$, where $\alpha, b > 0$. We derive a convergence bound using Theorem~\ref{theo_accstepsde} together with Theorem ~\ref{acc_ode_conv}, that $\dot{a}_t \leqslant \beta\frac{a_t}{t}$, where $\beta \leqslant \min(\frac{2b-\alpha}{3}, 2 - \alpha)$:
\begin{equation*}
    \mathbf{E}[f(X_t) - f_\star] \leqslant \frac{\beta^2}{t^\beta}\|x_0 -x_\star\|^2 + \frac{\gamma}{4t^\beta}\int_0^t \frac{s^{\beta}}{(s+1)^{ \alpha}}\mathbf{E}{\rm Tr}(\Sigma(X_s))ds.
\end{equation*}
On the one hand, the smaller the step sizes, the better the convergence for the term that contains the initial conditions. On the other hand, under bounded covariance $\Sigma(X_t) \preccurlyeq \Sigma$, the variance term behaves as $\mathcal{O}(\frac{1}{t^\beta})$ if $\beta \leqslant \alpha - 1$, and as $\mathcal{O}(\frac{1}{t^{\alpha-1}})$ otherwise (convergence requiring then $\alpha \leqslant 1$ and $\beta \leqslant 1$). For Nesterov's accelerated gradient flow with $b=3$, we have $\beta = 2 - \alpha \leqslant \alpha - 1$, and therefore $\alpha \geqslant \frac{3}{2}$. Taking $\alpha = \frac{3}{2}$, a convergence bound is given by:
\begin{equation*}
    \mathbf{E}[f(X_t) - f_\star] \leqslant \frac{9}{4\sqrt{t}}\|x_0 -x_\star\|^2 + \frac{\log{t}}{\sqrt{t}}\gamma{\rm Tr}(\Sigma).
\end{equation*}
We retrieve the result from Corollary~\ref{pr_step_size_tr_avg} for the SDE approximating SGD under Polyak-Ruppert averaging. Yet this result requires smaller step sizes $(\alpha \geqslant \frac{3}{2})$. It does not seem possible to accelerate SGD with diminishing step sizes. Ghadimi and Lan~\cite[Corollary 3]{Ghadimi15anoptimal} proved a convergence bound for a stochastic accelerated gradient method with  $\beta=2$ and $\alpha = \frac{1}{2}$, that we do not retrieve. Yet, in their approach, functions are minimized over a compact convex domain, whereas our approach focuses on an unbounded domain.

\section{Conclusion and future work}
\label{sec:future_work}
We have developed a systematic approach for finding quadratic Lyapunov functions for families of ODEs and SDEs approximating SGD. Verifying such a Lyapunov function is cast as verifying the feasibility of a small-sized LMI. From this formulation, it is possible to efficiently search for quadratic Lyapunov functions for arriving to convergence bounds. While we retrieve convergence guarantees similar to those of discrete-time systems, continuous-time models require less assumptions on the problem classes and can be analyzed through shorter proofs.

While obtaining guarantees for stochastic optimization methods might be tedious, the SDE approach allows for simpler analyses of the trade-off between the variance term and the term that forgets the initial conditions. A shortcoming of this approach is that this analysis does not include approximation guarantees between optimization methods and their continuous-time counterparts. In the deterministic setting, stability techniques are often developed to quantify this approximation efficiently~\cite{Gautschi1997}. In stochastic analysis, stochastic modified equation have been introduced by Li et~al.~\cite[Theorem 1]{li2017stochastic} to better approximate SGD, and stochastic methods with momentum. For non-convex functions, some analyses have also been done by Shi et~al~\cite{shi2020learning}. However, these approximation theorems often require many assumptions on the class of functions, which we believe could be further simplified using computer-assisted proofs.

Concerning possible extensions, this work relies on a specific family of quadratic Lyapunov functions provided by~\eqref{lyapunov_order2} and~\eqref{lyap_conv}. Whereas it was sufficient for our purposes, it turns out that it is possible to extend it by taking into account terms of the forms $\int_0^t f(X_s)ds$ or integrals of the quadratic terms that are used in our Lyapunov functions. While we did not consider those terms here, they could be useful for studying other methods or in other settings. The attentive reader might also have realized that the PEP technique for continuous-time systems can be applied without difficulty to differential~\cite[Section 3.2]{bolte2009} and monotone inclusions~\cite{Bot2016, Bot2017} problems. For monotone inclusions, interpolation results for casting the PEPs as semidefinite programs can be found in~\cite{ryu2020}. Finally, we note that it is still not clear how to use PEP-related techniques for directly studying higher-order methods and assumptions (already appearing in the variance term in the stochastic setting), which are also common for continuous-time systems~\cite{ATTOUCH2016}, or equivalents in the monotone inclusion setting~\cite{Chavdadora2021, bot2022}. The problem here is the lack of a clean performance estimation reformulation, beyond the somewhat indirect approach by~\cite{de2020worst}. We leave those investigations for future work.

\section*{Acknowledgments}
This work was funded by MTE and the Agence Nationale de la Recherche as part of
the ``Investissements d’avenir'' program, reference ANR-19-
P3IA-0001 (PRAIRIE 3IA Institute). We also acknowledge
support from the European Research Council (grant SEQUOIA
724063).

\bibliographystyle{siamplain}
\bibliography{references}

\appendix
\section{Proof for ODEs}

\subsection{Proofs for Theorem~\ref{conv_gf_conv}}
\label{sec:A0_conv}
\begin{proof}
Following the same procedure as for Theorem~\ref{conv_gf_mu} under strong-convexity assumption, and given a quadratic Lyapunov function $\mathcal{V}_{a_t, c_t}(X_t,t) = a_t \left(f(x_t) - f_\star \right) + c_t\|X_t - x_\star\|_2^2$ with $a_t, c_t\geqslant 0$, the maximization problem can be formulated into a semidefinite program
\begin{equation*}
    \begin{aligned}
      0 \geqslant \max_{G \succcurlyeq 0, F \in \mathbf{R}^2} & b_0^\top F + {\rm Tr}(A_0G), \\
      \text{subject to }
    & b_i^\top F + {\rm Tr}(A_iG) \geqslant 0, \ i \in \{1, 2 \}
\end{aligned}
\end{equation*}
where $A_0 = \begin{pmatrix}\dot{c}_t & -c_t \\ -c_t & -a_t \end{pmatrix}$, $A_1 = \begin{pmatrix}0 & 1/2 \\ 1/2 & 0 \end{pmatrix}$, $A_2= \begin{pmatrix}0 & 0 \\ 0 & 0\end{pmatrix}$, $b_0 = \dot{a}_t[1, \ -1]^\top $ $b_1 = [-1,\ 1]^\top $ and $b_2=[1, \ -1]^\top $, whose Lagrangian dual is given by the feasibility problem,
\begin{equation*}
    \begin{aligned}
     \min_{\lambda_t^{(1)}, \lambda_t^{(2)} \geqslant 0} & 0 \ \
      \text{s. t.}  \ S = \begin{pmatrix} \dot{c}_t & -c_t + \frac{\lambda^{(1)}_t}{2} \\ -c_t + \frac{\lambda^{(1)}_t}{2} & -a_t \end{pmatrix} \preccurlyeq 0, \ \dot{a}_t = \lambda^{(1)}_t - \lambda^{(2)}_t.
\end{aligned}
\end{equation*}
This formulation is exactly the LMI feasibility problem from the statement of Theorem~\ref{conv_gf_conv}. 
\smallbreak

\end{proof}

\subsection{Proof for Corollary~\ref{coro_accparam}}
\label{sec:A00}
\begin{proof}
We prove the claim in the convex case; the strongly convex setting is deduced by assuming an exponential form for the parameters $a_t = ae^{t\tau}$, and $P_t=Pe^{t\tau}$, where $\tau >0$ is a (linear) convergence rate to be determined.

From Theorem~\ref{acc_ode_conv}, verifying the Lyapunov function can be framed as verifying the feasibility of an LMI. That is, the desired Lyapunov function can be deduced from the existence of  $\lambda^{(1)}_t, \lambda^{(2)}_t \geqslant 0$ such that:
    \begin{equation*}
    \begin{aligned}
S=&\begin{pmatrix} -\frac{\mu}{2}(\lambda^{(1)}_t + \lambda^{(2)}_t) + \dot{p}_{t}^{(11)} & p^{(11)}_t - \beta_tp^{(12)}_t + \dot{p}^{(12)}_t & - p^{(12)}_t + \frac{\lambda_t^{(1)}}{2} \\ p^{(11)}_t - \beta_tp^{(12)}_t + \dot{p}^{(12)}_t & 2(p^{(12)}_t - \beta_tp^{(22)}_t) + \dot{p}^{(22)}_t & -p^{(22)}_t + \frac{a_t}{2} \\ - p^{(12)}_t + \frac{\lambda_t^{(1)}}{2} & -p^{(22)}_t + \frac{a_t}{2} & 0\end{pmatrix} \preccurlyeq 0, \\ &\dot{a}_t = \lambda^{(1)}_t - \lambda^{(2)}_t.
\end{aligned}
\end{equation*}
Because of the zero diagonal term in $S = (s_{ij})_{1 \geqslant ij\geqslant 3} \preccurlyeq 0$, it follows that $p^{(22)}_t = \frac{a_t}{2}$ and $p^{(12)}_t = \frac{\dot{a}_t}{2}$ (otherwise, the submatrix $\begin{pmatrix} s_{22} & s_{23} \\ s_{23} & 0 \end{pmatrix} \preccurlyeq  0$ has a strictly nonpositive determinant, which is a contradiction).

Let us now assume that $\lambda_t^{(2)}=0$, which leads to $\dot{a}_t = \lambda_t^{(1)}$. Because of the null entries, the matrix $S$ can be reduced to a $2\times2$ semi-definite negative matrix. Such a matrix has nonpositive diagonal terms (since $s_{11}s_{22} \geqslant (s_{12})^2$ and $s_{11} + s_{22} \geqslant 0$), that are $ \dot{p}_{t}^{(11)}  \leqslant 0$ and $2(p^{(12)}_t - \beta_tp^{(22)}_t) + \dot{p}^{(22)}_t$. which simplifies into $\dot{a}_t \leqslant \frac{2}{3}\beta_ta_t$.

For all $\epsilon > 0$, after integration of $\dot{a}_t \leqslant \frac{2}{3}\beta_ta_t$  between $\epsilon$ and $t$, that $a_t \leqslant a_\epsilon e^{\int_{\epsilon}^t  \frac{2}{3}\beta_sds}$. Therefore, $a_t \leqslant \lim_{\epsilon \rightarrow 0}a_\epsilon e^{\int_{\epsilon}^t  \frac{2}{3}\beta_sds}$. Recalling that $P_t$ is positive semidefinite, its determinant is nonnegative $p^{(11)}_tp^{(22)}_t - (p^{(12)}_t)^2 \geqslant 0$, that is $p^{(11)}_t \geqslant \frac{(p^{(12)}_t)^2}{p^{(22)}_t} = \frac{(\dot{a}_t)^2}{2a_t}$. After integration between $0$ and $t$, we obtain the following bound on $a_t$: $\sqrt{a_t} \leqslant \sqrt{a_0} + \frac{\sqrt{p_0^{(11)}}}{2}t$.

In other words, $a_t \leqslant \min\left((\sqrt{a_0} + \frac{\sqrt{p_0^{(11)}}}{2}t)^2, \lim_{\epsilon \rightarrow 0}a_\epsilon e^{\int_{\epsilon}^t  \frac{2}{3}\beta_sds}\right)$.
\end{proof}

\section{Proofs for SDEs}
\label{sec:A1}

\subsection{Proof for Theorem~\ref{pr_avg_conv}}
\label{sec:A11}

\begin{proof}

We rewrite the SDE into 
\begin{equation*}
\begin{aligned}
    \begin{pmatrix} dX_t \\ d\bar{X}_t\end{pmatrix} =& \left(\begin{pmatrix} 0 & 0 \\ \frac{1}{t} & -\frac{1}{t} \end{pmatrix}\otimes I_d\right)\begin{pmatrix} X_t \\ \bar{X}_t\end{pmatrix}dt +\left(\begin{pmatrix}- h_t & 0 \\ 0  & 0  \end{pmatrix}\otimes I_d\right)\begin{pmatrix} \nabla f(X_t) \\ \nabla f(\bar{X}_t) \end{pmatrix}dt \\
    &+\left(\begin{pmatrix}h_t (\gamma\Sigma(X_t))^{1/2} \\ 0\end{pmatrix}\otimes I_d\right)dB_t
\end{aligned}
\end{equation*}
 and denote $Y_t = \begin{pmatrix} X_t \\ \bar{X}_t \end{pmatrix}$. 
 \smallbreak
 We consider the quadratic function~\eqref{lyap_avg} $\mathcal{V}_{a^{(1)}_t, a^{(2)}_t, P_t}(X_t, \bar{X}_t ,t) = \mathcal{V}(Y_t,t)$. Applying Ito's formula, its derivative with respect to time is $\frac{d}{dt} \mathcal{V}(Y_t,t) = \frac{\partial}{\partial t}\mathcal{V}(Y_t,t) + \frac{\partial}{\partial y}\mathcal{V}\frac{dY_t}{dt} + \frac{1}{2}\gamma h^2_t {\rm Tr}[\frac{\partial^2}{\partial Y_t^2}\mathcal{V}(Y_t,t)^\top \begin{pmatrix}\Sigma(X_t) \\ 0\end{pmatrix}]$. By taking the expectation of Ito's formula, we have $\frac{d}{dt} \mathbf{E}\mathcal{V}(Y_t,t) = \mathbf{E}[\frac{\partial}{\partial t}\mathcal{V}(Y_t,t) + \frac{\partial}{\partial y}\mathcal{V}\frac{dY_t}{dt}] +\frac{1}{2}\gamma h^2_t \mathbf{E}{\rm Tr}(\frac{\partial^2}{\partial Y_t^2}\mathcal{V}(Y_t,t)^\top \begin{pmatrix}\Sigma(X_t) \\ 0\end{pmatrix})$.
 \smallbreak
 The variance term $\frac{1}{2}\gamma h^2_t \mathbf{E}{\rm Tr}(\frac{\partial^2}{\partial Y_t^2}\mathcal{V}(Y_t,t)^\top \begin{pmatrix}\Sigma(X_t) \\ 0\end{pmatrix})$ depends on the second derivative of $f$ in the space variable $(X_t)$, and on the covariance $\Sigma_t$. We do not take this term into account in the performance estimation frawework.
 \smallbreak
 We formulate the problem of verifying a quadratic function as verifying the inequality $\frac{d}{dt}\mathbf{E}\mathcal{V}(X_t, \bar{X}_t,t) \leqslant \frac{1}{2}\gamma h^2_t \mathbf{E}{\rm Tr}(\frac{\partial^2}{\partial X_t^2}\mathcal{V}(X_t, \bar{X}_t, t)^\top \Sigma(X_t))$ (that is $\mathbf{E}[\frac{\partial}{\partial t}\mathcal{V}(Y_t,t) + \frac{\partial}{\partial y}\mathcal{V}\frac{dY_t}{dt}] \leqslant 0$) holds for any twice continuously differentiable function $f \in \mathcal{F}_{0, \infty}$, and any trajectory $(X_t, \bar{X}_t)$ generated by the SDE under Polyak-Ruppert averaging~\eqref{ODE_mean1}. 
 \smallbreak
 This problem is equivalent with verifying that $\frac{d}{dt}\mathcal{V}(\tilde{Y}_t,t) \leqslant 0$ holds for any function twice continuously differentiable $f \in \mathcal{F}_{0, \infty}$ and any trajectory $\tilde{Y}_t = \begin{pmatrix} \tilde{X}_t & \tilde{\bar{X}}_t\end{pmatrix}^\top $ generated by deterministic gradient flow from the SDE~\eqref{ODE_mean1} with $\gamma = 0$. 
\smallbreak
We follow the methodology developed in Section~\ref{sec:deterministic} for ODEs. We formulate the verification of such a Lyapunov function as the maximization problem
\begin{equation*}
    \begin{aligned}
      0\geqslant \max_{X_t \in \mathbf{R}^d,  d\in \mathbf{N}, \ f\in \mathcal{F}_{0, \infty}} \ & \ \frac{d}{dt}  \mathcal{V}_{a^{(1)}_t, a^{(2)}_t, P_t}(X_t, \bar{X}_t, t)\\
      \text{subject to }
    \ \dot{X}_t &= -h_t\nabla f(\bar{X}_t), \
        \dot{\bar{X}}_t = \frac{X_t - \bar{X}_t}{t}dt.
\end{aligned}
\end{equation*}
This maximization problem can be formulated as the following SDP program

\begin{equation*}
    \begin{aligned}
    & \min_{\lambda_t^{(i)} \geqslant 0, \ i \in \{1, ..., 6 \}} \ 0 \, \\
&\begin{pmatrix} \dot{p}^{(11)}_t + \frac{2p^{(12)}_t}{t} & \dot{p}^{(12)}_t + \frac{p^{(22)}_t - p^{(12)}_t}{t}& \frac{\lambda^{(6)}_t + \lambda^{(4)}_t - 2h_tp^{(11)}_t}{2} & \frac{a^{(2)}_t}{2t} - \frac{\lambda^{(5)}_t}{2}\\ 

\dot{p}^{(12)}_t + \frac{p^{(22)}_t - p^{(12)}_t}{t} & \dot{p}^{(22)}_t - \frac{2p^{(22)}_t}{t} & -\frac{\lambda^{(6)}_t - 2h_tp^{(12)}_t}{2} & -\frac{a^{(2)}_t}{2t} + \frac{\lambda^{(5)}_t + \lambda^{(3)}_t}{2}\\ 

 \frac{\lambda^{(6)}_t + \lambda^{(4)}_t - 2h_tp^{(11)}_t}{2} & -\frac{\lambda^{(6)}_t - 2h_tp^{(12)}_t}{2} & -a^{(1)}_t & 0 \\ 

\frac{a^{(2)}_t}{2t} - \frac{\lambda^{(5)}_t}{2} & -\frac{a^{(2)}_t}{2t} + \frac{\lambda^{(5)}_t + \lambda^{(3)}_t}{2} & 0 &  0\end{pmatrix} \preccurlyeq 0, \\ 
&\dot{a}^{(1)}_t + \lambda^{(1)}_t + \lambda^{(5)}_t = \lambda^{(4)}_t + \lambda^{(6)}_t, \\ 
&\dot{a}^{(2)}_t + \lambda^{(2)}_t + \lambda^{(6)}_t = \lambda^{(3)}_t + \lambda^{(5)}_t + \dot{a}^{(1)}_t. 
\end{aligned}
\end{equation*}
Note that the Gram matrix $G$ and function value vector $F$ to be introduced are different from the setting explored in Section~\ref{sec:deterministic}, more precisely $G = Q^\top Q \succcurlyeq 0, \ Q = [X_t - x_\star, \bar{X}_t - x_\star, g_t, \bar{g}_t]$ and $F = (f(X_t), f(\bar{X}_t), f_\star)$.

The final inequality follows from Ito's formula.
\end{proof}

\subsection{Proof for Theorem~\ref{acc_sde_theorem}}
\label{sec:A21}
\begin{proof}
From the reasoning from Theorem~\ref{conv_gf_conv} and Theorem~\ref{pr_avg_conv}, worst-case guarantee can be formulated into a maximization problem:
\begin{align*}
     \max_{f \in \mathcal{F}_{0,\infty},  d\in \mathbf{N}, X_t \in \rm{R}^d} \ & \ \frac{d}{dt}\mathcal{V}_{a^{(1)}_t, a^{(2)}_t, P_t}(X_t, \bar{X}_t, t), \\
    \text{s.t. } \ & d^2X_t + \beta_t dX_t + \nabla f(X_t)dt + \sqrt{\gamma \Sigma}dB_t = 0, \\
     & d\bar{X}_t = \frac{X_t - \bar{X}_t}{t}dt.
\end{align*}
A control of this quantity can be achivied with the following LMI, for $\lambda_t^{(i)} \geqslant 0, \ i \in \{1, ..., 6 \}$ and where $*$ corresponds to the symmetric entries of the matrix, and where $\beta_t = \frac{3}{t}$,
\begin{align*}
    &\begin{pmatrix}
    \dot{p}_t^{(11)}  & \dot{p}_t^{(12)} - \beta_t p_t^{(12)}  & \dot{p}_t^{(13)} - \frac{p_t^{(13)}}{t} -  & -p_t^{(11)} & 0 \\
    - 2\beta_t p_t^{(11)}  & +  \frac{p_t^{(13)}}{t} + p_t^{(22)} & p_t^{(12)}\beta_t + p_t^{(23)} & & \\
    * & \dot{p}_t^{(22)} + 2 \frac{p_t^{(22)}}{t} & p_t^{(23)} - 2 \frac{2p_t^{(22)}}{t}   & a_t^{(1)}- p_t^{(12)}  & \frac{a_t^{(2)}}{2t} - \frac{\lambda_t^{(5)}}{2} \\
    & & + \frac{p_t^{(33)}}{t} & + \frac{\lambda_t^{(4)} + \lambda_t^{(6)}}{2} & \\
    * & * & \dot{p}_t^{(33)} -  \frac{2p_t^{(33)}}{t} & -\frac{\lambda_t^{(6)}}{2} - p_t^{(13)}& \frac{\lambda_t^{(5)} + \lambda_t^{(3)} - a_t^{(2)}/t}{2} \\
    * & * & * & 0 & 0 \\
    * & * & * & * & 0
    \end{pmatrix} \preccurlyeq 0, \\
    &\lambda_t^{(6)} + \lambda_t^{(4)} = \lambda_t^{(5)} + \lambda_t^{(1)} +  \dot{a}_t^{(1)},\
    \lambda_t^{(2)} + \lambda_t^{(6)} + \dot{a}_t^{(2)} = \lambda_t^{(3)} + \lambda_t^{(5)}.
\end{align*}
Thus, $p_t^{(11)} = 0 $, and because $P_t$ is positive semidefinite, $p_t^{(12)} = p_t^{(13)} = 0$. If in addition $a_t^{(1)}=0$, the unique feasible Lyapunov function is $\mathcal{V}_{a^{(1)}_t, a^{(2)}_t, P_t}=0$. Otherwise, the LMI can be simplified to the LMI of Theorem~\ref{acc_ode_conv} and the Lyapunov function follows from Corollary~\ref{coro_accparam}. 
\end{proof}

\end{document}